\documentclass{article}

\usepackage{arxiv}

\usepackage[utf8]{inputenc} 
\usepackage[T1]{fontenc}    
\usepackage{hyperref}       
\usepackage{url}            
\usepackage{booktabs}       
\usepackage{amsfonts}       
\usepackage{nicefrac}       
\usepackage{microtype}      
\usepackage{lipsum}
\usepackage{graphicx}
\usepackage{amsmath}
\usepackage{bm}
\usepackage{algorithm}
\usepackage{algpseudocode}
\usepackage{amsthm}
\usepackage{subfigure}
\usepackage{multirow} 
\graphicspath{ {./images/} }

\def\d{\mathrm{d}}
\newtheorem{assumption}{Assumption}
\newtheorem{theorem}{Theorem}
\newtheorem{proposition}{Proposition}
\newtheorem{definition}{Definition}

\title{Two Kinds of Learning Algorithms for Continuous-Time VWAP Targeting Execution}

\author{
 Xingyu Zhou \\
  School of Mathematical Sciences\\
  Fudan University\\
   Shanghai, China 200433 \\
  \texttt{xingyuzhou22@m.fudan.edu.cn} \\
   \And
 Wenbin Chen \\
  School of Mathematical Sciences\\
  Fudan University\\
  Shanghai, China 200433 \\
  \texttt{wbchen@fudan.edu.cn} \\
  \And
Mingyu Xu \\
  School of Mathematical Sciences\\
   Fudan University\\
  Shanghai, China 200433 \\
  \texttt{xumy@fudan.edu.cn} \\
}

\begin{document}
\maketitle
\begin{abstract}
The optimal execution problem has always been a continuously focused research issue, and many reinforcement learning (RL) algorithms have been studied. In this article, we consider the execution problem of targeting the volume weighted average price (VWAP) and propose a relaxed stochastic optimization problem with an entropy regularizer to encourage more exploration. We derive the explicit formula of the optimal policy, which is Gaussian distributed, with its mean value being the solution to the original problem. Extending the framework of continuous RL to processes with jumps, we provide some theoretical proofs for RL algorithms. First, minimizing the martingale loss function leads to the optimal parameter estimates in the mean-square sense, and the second algorithm is to use the martingale orthogonality condition. In addition to the RL algorithm, we also propose another learning algorithm:  adaptive dynamic programming (ADP) algorithm, and verify the performance of both in two different environments across different random seeds. Convergence of all algorithms has been verified in different environments, and shows a larger advantage in the environment with stronger price impact. ADP is a good choice when the agent fully understands the environment and can estimate the parameters well. On the other hand, RL algorithms do not require any model assumptions or parameter estimation, and are able to learn directly from interactions with the environment.
\end{abstract}


\section{Introduction}
In the last decade, algorithmic trading and quantitative finance have been booming, and various trading strategies, such as market making strategies, pair trading and optimal execution of large orders \cite{cartea2015algorithmic}, have attracted more and more people's attention. Because it is not possible to short stocks in China, a unique intraday strategy is attracting growing interest, which is trading on the basis of the previous day's position to get the enhancement of income. VWAP is the benchmark for evaluating intraday prices, and the research on VWAP in investment is very important in China. In the Chinese market, algorithmic trading has become the choice of more and more investors. How to execute a large order is of great importance in research and trading \cite{gueant2016financial}. Generally speaking, the designed algorithm decomposes large orders into smaller ones, which will be executed sequentially within the given trading time. Under the same conditions, the longer trading transaction time causes smaller impacts on the market, but the execution price is more uncertain. 
		
		The discussion of optimal execution can be traced back to the work of Bertsimas and Lo in 1998, which shows the optimal policy is to split the original order into smaller orders in time, under the assumption that price is a random walk process with linear permanent price impacts \cite{bertsimas1998optimal}. The model considered by Almgren in 2001 is the execution of portfolios, within which framework price impacts are categorized into temporary and permanent components when considering risk factors \cite{almgren2001optimal}.  Almgren then considered a model with nonlinear price impacts \cite{almgren2003optimal} and conducted an empirical study in the US market\cite{almgren2005direct}. Schied and Sch{\"o}neborn analyzed the sensitivity of the value function and optimal strategies to various model parameters in the liquidity model of Almgren \cite{schied2009risk}. 
		Subsequently, there are many different studies on optimal execution problems, such as price impacts (Gerig and Austin \cite{gerig2008theory}, Moro et al. \cite{moro2009market}, Alfonsi et al. \cite{alfonsi2010optimal}, Gatheral et al. \cite{gatheral2012transient}, Cartea and Jaimungal \cite{cartea2016incorporating}, Curato et al. \cite{curato2017optimal}), using limit orders in trading (Gu{\'e}ant, Lehalle et al. \cite{gueant2012optimal}, Cartea et al. \cite{cartea2014buy}, Cartea and Jaimungal \cite{cartea2015optimal}), different risk measurement methods (Kharroubi and Pham \cite{kharroubi2010optimal}, Robert et al. \cite{robert2012measuring}, Gu{\'e}ant and Lehalle \cite{gueant2015general}, Cartea and Jaimungal \cite{cartea2015risk}, Tsoukalas et al. \cite{tsoukalas2019dynamic}.), and different assumptions about the price process (Bayraktar and Ludkovski \cite{bayraktar2014liquidation}, Moazeni et al. \cite{moazeni2013optimal}, Alfonsi and Blanc \cite{alfonsi2016dynamic}). 
		
		In optimal execution problems, traders consistently place one-sided orders, which can potentially lead to an unnatural accumulation of one-sided pressure in the order book. Thus, the most desirable approach for traders would be to adopt stealthy strategies that allow their orders to integrate into the order flow. VWAP was introduced by Berkowitz \cite{berkowitz1988total} in 1988 and is now a common benchmark used in the financial industry to measure price impacts by comparing it to execution prices. A static VWAP strategy was proposed by Konishi in 2002 \cite{konishi2002optimal}, and many dynamic predicting VWAP algorithms for volume decomposition have been proposed since then (see, for example, Białkowski \cite{bialkowski2008improving}, Humphrey \cite{humphery2011optimal}, McCulloch and Kazakov \cite{mcculloch2012mean}). There are also some machine learning algorithms that have been applied to VWAP prediction (see, for example, Alvim et al. \cite{alvim2010daily}, Brownlees et al.\cite{brownlees2011intra}, Liu and Lai \cite{liu2017intraday}, Jeong et al. \cite{jeong2021using}). 
		
		Another popular method is to design VWAP algorithms from the perspective of optimal control problems. The basic idea of this approach is that the market is modeled with different stochastic processes, and then different value functions are set to measure policies according to the investor's objectives. Under appropriate conditions, the optimal value function can be expressed by the Hamilton-Jacobi-Bellman (HJB) equation, and the optimal strategy can be solved by dynamic programming. See the book of Yong and Zhou \cite{yong1999stochastic} for more about stochastic control and HJB equations. 
		Frei and Westray aimed to minimize the sliding point relative to the market VWAP, derived and solved the corresponding HJB equation \cite{frei2015optimal}. In a similar way, Guéant used an exponential function as the utility function, while optimizing the execution cost and difference between the execution price and VWAP \cite{gueant2014vwap}. In addition, Mitchell established a model of interaction between volumes and prices, and obtained the optimal dynamic VWAP strategy within his more general framework \cite{mitchell2013optimal}. 
		
		Reinforcement learning (RL) is a hot research field in recent years, which mainly studies how agents should take actions in the interaction with environment to maximize the cumulative rewards. Many significant applications in finance have been found in this field, particularly in constructing portfolios (Wang et al. \cite{wang2019alphastock}, Wang et al. \cite{wang2021deeptrader}), capturing single-asset trading signals (Gao et al. \cite{gao2000algorithm}, Deng et al. \cite{deng2016deep}, Almahdi and Yang \cite{almahdi2017adaptive}), option hedging (Buehler et al. \cite{buehler2019deep}, Kolm and Ritter \cite{kolm2019dynamic}), and optimal execution. Using Deep Q-network (DQN) algorithm, Nevmyvaka completed empirical demonstrations of RL in modern financial markets in the US \cite {nevmyvaka2006reinforcement}. Park and Van Roy proposed a method for simultaneously trading and learning model parameters \cite {park2015adaptive}. Furthermore, some new RL algorithms have also been applied to optimal execution problems, such as Double DQN (Dab{\'e}riu et al. \cite{daberius2019deep}, Karpe et al. \cite{karpe2020multi}, Ning et al. \cite{ning2021double}), Proximal Policy Optimization (Lin and Beling \cite{lin2021end}, Fang \cite{fang2021universal}, Schnaubelt \cite{schnaubelt2022deep}). 
		Haarnoja et al. proposed an off-policy Actor-Critic deep RL algorithm in 2018, named the Soft Actor-Critic (SAC), which is based on the maximum entropy RL framework. In SAC, the actor aims to maximize both the expected reward and the entropy \cite{haarnoja2018soft} .
		
		Despite the significant achievements of RL in solving optimal execution problems, traditional RL methods are still limited to discrete states and action spaces. However, it might be more practical to concern about continuous states and action spaces. In recent years, there has been a growing interest in continuous RL. It is natural to establish a continuous time framework for high-frequency decision problems, such that some unique results can be used to develop new RL algorithms. In 2020, Wang et al. considered continuous-time RL and devised a method to encourage exploration while also adding an entropy regularizer to balance exploration and exploitation, and transform original optimization problems into relaxed stochastic control problems \cite {wang2020reinforcement}. Subsequently, Wang et al. applied this method to the continuous-time mean-variance problem, proving that the optimal strategy under the regularized exploratory formulation must be Gaussian distribution \cite{wang2020continuous}. Jia and Zhou have published their series of work, establishing a general theoretical framework for continuous-time RL. Jia and Zhou first proposed a unified framework to study policy evaluation and related temporal difference methods in continuous state and action spaces, and designed two algorithms about the martingale loss function and the martingale orthogonality condition, which may correspond to classical RL algorithms \cite{jia2022policy}. Then Jia and Zhou demonstrated that by introducing an auxiliary running reward function, the gradient of the value function can be calculated similarly through the method in \cite{jia2022policy}, which allows the policy to be updated \cite{jia2022policy2}. In 2023, q-learning corresponding to discrete-time Q-learning has been proposed. Specifically q-functions are defined related to the instantaneous advantage rate function and Hamiltonian, and also able to be characterized by martingale conditions \cite{jia2023q}. Considering Almgren-Chriss model in continuous time, Wang, Gao and Li have developed an offline Actor-Citic algorithm for the optimal execution problem based on the framework of \cite{jia2022policy} in 2023, and provided finite-time error analysis  \cite{wang2023reinforcement}. There are also some other researches on continuous RL, such as mean-field games (Guo et al. \cite{guo2022entropy}, Firoozi and Jaimungal \cite{firoozi2022exploratory}, Frikha et al. \cite{frikha2023actor}), regret bounds for Markov decision processes (Gao and Zhou \cite{gao2022logarithmic, gao2022square}, Bayraktar and Kara \cite{bayraktar2023approximate}), and linear-quadratic problems (Szpruch et al. \cite{szpruch2021exploration}, Li et al.\cite{li2022stochastic}, Hambly et al. \cite{hambly2023policy}). 
		
		In this paper, aimed at solving the VWAP targeting execution problem, we consider an exploratory version, prove the existence and uniqueness of state equations and show that the optimal policy follows a Gaussian distribution by solving the HJB equation of the exploratory problem. Then
		we propose an adaptive dynamic programming (ADP) method for obtaining both deterministic and stochastic policies, which is the first time that ADP is used to solve VWAP targeting problems as far as we know. As for RL algorithms, following the framework created by Jia and Zhou \cite{jia2022policy,jia2022policy2}, we extend it to stochastic processes with Poisson jumps, in which jumps are used to characterize the market trading speed. Specifically, we demonstrate two martingale approaches: martingale loss function and martingale orthogonality condition, both applicable to policy evaluation. Furthermore, we show how the martingale orthogonality condition can be reformulated as an equivalent minimization problem. We also prove that the policy gradient can be calculated directly from samples by introducing an auxiliary reward function. Based on the theoretical derivation, we design two Actor-Critic algorithms combining policy evaluation and policy improvement to solve the VWAP targeting problem, called ML-AC and MO-AC. In the numerical results, we employ the Monte Carlo method to construct two distinct market simulators, adopt the time-weighted average price (TWAP) as a benchmark, and utilize the Soft Actor-Critic algorithm \cite{haarnoja2018soft} for comparative analysis. This algorithm is a type of RL algorithm applicable in discrete spaces with an entropy regularizer. According to the numerical results, we verify the effectiveness of our algorithms, and that ADP is a good choice when the agent knows the environment well, while Actor-Critic algorithms are more flexible to directly learn from interaction with the environment. 
		
		The work of our paper is described as follows. First, the liquidation problem considers market trading speed as a mean-reverting process with Poisson jumps, and the stock midprice process as an arithmetic Brownian motion. The midprice is subject to permanent price impacts from agents, and the transaction price is subject to temporary price impacts. For more modeling details refer to \cite{cartea2015algorithmic} and its references. The objective is to maximize the return as well as to track the speed of market trading during the process, and the explicit expression for the optimal policy can be solved in this model setup, which depends on some parameters, see for \cite{cartea2016closed}. The parameters of the jump diffusion process are not amenable to direct calculation via maximum likelihood estimation, primarily due to the absence of an explicit likelihood function for the model's process. Instead, we may only obtain numerical estimates using methods like the inverse Fourier transformation of relevant characteristic functions \cite{abate1995numerical}. To avoid this problem, we consider the HJB equation directly and solve it through ADP proposed by Werbos \cite{werbos1992approximate}, and the same approach can be applied to the exploratory liquidation problem. One advantage of ADP is that we do not need complex numerical calculation to estimate the parameters of market trading speeds.
		
		Furthermore, to find the optimal policy with randomness, we consider the exploratory liquidation problem and introduce an entropy regularizer to measure our level of exploration. We prove that there exists a unique solution for exploratory state equations and the optimal policy obeys a Gaussian distribution. The mean value of this distribution is the same as the optimal policy of the original problem, and the standard deviation dependents on the temperature parameter, which represents the level of encouragement for exploration. We then refer to the RL framework introduced by Jia and Zhou  \cite{jia2022policy,jia2022policy2}, and extend it to stochastic processes with Poisson jumps to design Actor-Critic (AC) algorithms for continuous-time VWAP targeting execution problem. As for RL algorithms, we have theoretically proven that the value function can be characterized by minimizing the martingale loss function and solving the equation of martingale otrhogonality condition. Instead of using stochastic approximation method, whose convergence may be difficult since we have no knowledge of the equation near the true solution
		, we make some modifications to solve an equivalent minimization problem, also provide a pratical way to choose the test function and take the terminal condition into consideration. Combining policy evaluation and policy improvement, we design two Actor-Critic algorithms: ML-AC and MO-AC to solve the exploratory VWAP targeting execution problem. 
		
		At last, we simulate two different markets and test our algorithms, where the price impacts in the first market are smaller while the terminal penalty is larger. The settings for the second market are completely opposite. In both environments, the convergence of two kinds of learning algorithms to respective optimal policies is verified, with ADP initially exhibiting a relatively stable learning curve leading to increasing returns. As for two RL algorithms, despite some fluctuations at the beginning, the final average return is very close to which obtained by the optimal policy, and the policy learned through MO-AC is a distribution with less variation than ML-AC. Comparing the return gained by TWAP and SAC, the result shows that our methods have many advantages when there are larger price impacts in the market setting.
		
		This paper is organized as follows. In Section \ref{section_classical_problem}, we introduce the execution problem of targeting percentage of market trading speed and the explicit expression of the optimal solution. In Section \ref{section_DP}, we propose an adaptive dynamic programming algorithm to solve this problem. In Section \ref{section_exploratory}, we incorporate exploration into our problem and derive the optimal policy with randomness. Section \ref{section_RL} provides details of the principle of two RL algorithms. Section \ref{section_result} explains how we simulate the market parameters and create market simulators, and then shows some numerical results. At last, we give the conclusions in Section \ref{section_conclusion}.

\section{The problem of targeting percentage of market trading speed}
		\label{section_classical_problem}
		
		Throughout this paper, we use $ C $ with different subscripts to denote a constant that does not depend on the model variables, although the value represented by them may vary in different places. For a matrix $ A $, denote the Frobenius norm of $ A $ as $ |A| $. 
		
		We consider that the agent's execution policy targets a certain percentage of the speed at which other participants in the markets are trading. Our mission is to sell $ \mathfrak{N} $ shares of a stock within the time horizon from 0 to $ T $, where $ T $ is the terminal time of this liquidation problem. $W= \{W_t\}_{0\leq t\leq T} $ is a standard one-dimensional Brownian motion, $ N=\{N_t\}_{0\leq t\leq T} $ is a homogeneous Poisson process with intensity $\lambda$, which is a positive constant, and both are defined in the probability space $ (\Omega,\mathcal{F},\mathbb{P}^{W, N};\{\mathcal{F}_t^{W, N}\}_{0\leq t\leq T}) $. $ \{\mathcal{F}_t^{W, N}\}_{0\leq t\leq T} $ is the natural filtration generated by $ W $ and $ N $. We use $ \mathbb{E}^{\mathbb{P}^{W, N}} $ to indicate the expectation on this probability space. Set $ v=\{v_t\}_{0\leq t\leq T} $ to be the agent's trading policy, where $ v_t $ is the trading speed at $ t $. $ v_t $ belongs to the set of all possible trading speeds $\mathcal{A}$, and $ Q_t^v $ is the inventory which satisfies
		\begin{displaymath}
			\d Q^v_t=-v_t\d t.
		\end{displaymath} 
	Assume that the midprice process of the stock $ S^v_t $ is modeled by an arithmetic Brownian motion, and is affected by the permanent impact from the agent's trading,
	\begin{displaymath}
		\d S^v_t=-bv_t\d t+\sigma \d W_t,
	\end{displaymath}
	where $ b\geq0 $ is the parameter representing the permanent impact. Considering the transaction price $ \hat{S}_t^v $ for the agent, we also need to introduce the temporary impact, which is represented by $ k\geq0 $:
	\begin{displaymath}
		\hat{S}_t^v=S^v_t-kv_t. 
	\end{displaymath}
	Under the trading policy $v $, the cash process of the agent $X ^ v_ t $ satisfies
	\begin{displaymath}
		\d X^v_t=\hat{S}_t^vv_t\d t.
	\end{displaymath}
	The agent aims to determine an optimal liquidation speed, which is a percentage of the market trading speed. Denote the market trading speed by $\mu_t$, 
	\begin{equation}
		\d \mu_t=-\kappa\mu_t\d t+\eta_{1+N_{t-}} \d N_t,
		\label{trading_speed}
	\end{equation} 
	where $ N_t $ is a homogeneous Poisson process and $ \{\eta_1,\eta_2\cdots\}$ are non-negative independent and identically distributed random variables with the same distribution as $\eta$, whose moments are finite. The solution of (\ref{trading_speed}) is 
	\begin{align*}
		\mu_t&=e^{-\kappa t}\mu_0+\int_{0}^t\eta_{1+N_{u-}} dN_u=e^{-\kappa t}\mu_0+\sum_{m=1}^{N_t}e^{-\kappa(t-\tau_m)}\eta_m, 
	\end{align*}
	where $\tau_m$ denotes the stopping time of the $ m^{th} $ jump of the Poisson process. 
	
	The state is denoted by $ \mathbf{Y}^v=\{\mathbf{Y}^v_t\}_{0\leq t\leq T} $ under the policy $ v $, where $ \mathbf{Y}_t^v=(S_t^v,X_t^v,Q_t^v,\mu_t)^{\mathsf{T}} $, and then we represent above equations in vector form: \begin{equation}
		\d\mathbf{Y}_t^v=
		\begin{pmatrix}
			-bv_t\\(S_t^v-kv_t)v_t\\-v_t\\-\kappa\mu_t
		\end{pmatrix}\d t+\begin{pmatrix}
			\sigma\\0\\0\\0
		\end{pmatrix}\d W_t+\begin{pmatrix}
			0\\0\\0\\\eta_{1+N_{t-}}
		\end{pmatrix}\d N_t,
		\label{enviroment}
	\end{equation}
which is formulated by
\begin{displaymath}
	\d \mathbf{Y}_t^v=\mathbf{b}(t,\mathbf{Y}_t^v,v_t)\d t+\bm{\sigma} \d W_t+\bm{\eta}_{1+N_{t-}} \d N_t. 
\end{displaymath}
The value function of the agent consists of two parts. The first part is to maximize terminal wealth: $  X_T^v+Q^v_T(S^v_T-\alpha Q^v_T) $, where $\alpha$ represents the penalty caused by terminal transaction. The second part is to measure difference between the policy $ v $ and the market trading speed times $\rho$. Let the initial state denoted by $ \mathbf{y}=(x,S,\mu,q)^{\mathsf{T}}\in\mathbb{R}^4 $, the terminal function denoted by $ f(\mathbf{Y}^v_T)=X^v_T+Q^v_T(S^v_T-\alpha Q^v_T) $, and the running reward function denoted by $ r(\mathbf{Y}^v_t,v)=-\varphi(v-\rho\mu_t)^2 $. Given a policy $ v $, the corresponding value function is defined by
\begin{equation}
	H^v(t,x,S,\mu,q)=\mathbb{E}^{\mathbb{P}^{W, N}}_{t,x,S,\mu,q}\left[X_T^v+Q^v_T(S^v_T-\alpha Q^v_T)-\int_{t}^{T}\varphi(v_u-\rho\mu_u)^2\d u\right],
	\label{value_function}
\end{equation}
where the subscript represents calculating conditional expectation under the initial state $ (t,x,S,\mu,q) $ and $\varphi$ is the penalty parameter. 

It is easy to verify that the model setups satisfy the following assumptions, which ensure the well-posedness of this stochastic control problem. 
\begin{assumption}
	The following conditions hold true: 
	\begin{enumerate}
		\item $ \mathbf{b}(t,\mathbf{y},v) $ is continuous in its variables;
		\item $ \mathbf{b}(t,\mathbf{y},v) $ is uniformly Lipschitz continuous in $\mathbf{y}$, i.e. for $ \mathbf{y}_1, \mathbf{y}_2\in\mathbb{R}^4, $ $$\exists C>0,\quad |\mathbf{b}(t,\mathbf{y}_1,v)-\mathbf{b}(t,\mathbf{y}_2,v)|\leq C|\mathbf{y}_1-\mathbf{y}_2|,\qquad \forall (t,v)\in[0,T]\times \mathcal{A};$$
		\item $ \mathbf{b}(t,\mathbf{y},v) $ is linear increasing, i.e. for $ \mathbf{y}\in \mathbb{R}^4, $ 
		$$ \exists C>0,\quad |\mathbf{b}(t,\mathbf{y},v)|\leq C(1+|\mathbf{y}|),\qquad \forall (t,v)\in[0,T]\times \mathcal{A};$$
		\item $ r $ and $ f $ are polynomial growth, i.e. for $ \mathbf{y}\in\mathbb{R}^4, v\in\mathcal{A}, $
		$$ \exists C>0,\ p\geq1,\qquad  |r(\mathbf{y},v)|\leq C(1+|v|^p+|\mathbf{y}|^p),\qquad  |f(\mathbf{y})|\leq C(1+|\mathbf{y}|^p) ;$$ 
	\end{enumerate}
	\label{assumption1}
\end{assumption}

The target of the agent is to solve the following stochastic control problem: 
\begin{displaymath}
	\sup\limits_{v\in\mathcal{A}}\mathbb{E}^{\mathbb{P}^{W, N}}_{t,x,S,\mu,q}\left[X_T^v+Q^v_T(S^v_T-\alpha Q^v_T)-\int_{t}^{T}\varphi(v_u-\rho\mu_u)^2\d u\right],
\end{displaymath}
and the optimal value function is denoted by $ H(t,x,S,\mu,q) $. Applying the method of dynamic programming, we know that $H$ should satisfy the following PDE:
\begin{equation}
	\begin{aligned}
		0=\bigg(\partial_t+\frac{1}{2}\sigma^2\partial_{SS}&+\mathcal{L}^\mu\bigg) H\\
		&+\sup\limits_{v}\left\{(S-kv)v\partial_xH-v\partial_qH-bv\partial_SH-\varphi(v-\rho\mu)^2\right\},
	\end{aligned}
	\label{HJB}
\end{equation}
where $ \mathcal{L}^\mu $ is the infinitesimal generator of the Poisson process $ \mu $, which is
\begin{displaymath}
	\mathcal{L}^\mu H(t,x,S,\mu,q)=-\kappa\mu\partial_\mu H+\lambda\mathbb{E}[H(t,x,S,\mu+\eta,q)-H(t,x,S,\mu,q)], 
\end{displaymath} 
and the expectation is with respect to the random variable $\eta$. The terminal condition of $ H $ is 
\begin{displaymath}
	H(T,x,S,\mu,q)=x+q(S-\alpha q). 
\end{displaymath}
Observe that the item inside the curly brackets of (\ref{HJB}) is quadratic in $ v $, so we can obtain the optimal policy
\begin{equation}
	v^*=\dfrac{S\partial_xH-\partial_qH-b\partial_SH+2\varphi\rho\mu}{2(k\partial_xH+\varphi)}. 
	\label{best_speed}
\end{equation}
Cartea et al. \cite{cartea2016closed} found the explicit expression of the optimal policy
\begin{equation}
	v^*=\frac{1}{k+\varphi}\left\{\left[\varphi\rho\mu-\frac{1}{2}h_1(t,\mu)\right]-\left[\frac{1}{2}b+h_2(t,\mu)\right]q\right\}, 
	\label{best_speed_h}
\end{equation}
where 
\begin{align}
	&h_2(t,\mu)=-\left(\dfrac{T-t}{k+\varphi}+\frac{1}{\alpha-\frac{1}{2}b}\right)^{-1}-\frac{1}{2}b,\quad\ \  h_1(t,\mu)=l_0(t)+l_1(t)\mu,\nonumber\\
	&l_1(t)=2\varphi\rho((T-t)+\zeta)^{-1}\dfrac{1-e^{-\kappa(T-t)}}{\kappa},\qquad  \zeta=\dfrac{k+\varphi}{\alpha-\frac{1}{2}b},\nonumber\\
	&	l_0(t)=2\varphi\rho\lambda\mathbb{E}[\eta]((T-t)+\zeta)^{-1}\dfrac{e^{-\kappa(T-t)}-1+\kappa(T-t)}{\kappa^2}. \nonumber
\end{align}
The traditional method to approximate the optimal policy is estimating parameters needed in (\ref{best_speed_h}), but it's not easy because there is no explicit expression of the likelihood function for parameters in (\ref{trading_speed}). From another point of view, we consider the HJB (\ref{HJB}) directly and use dynamic programming method to solve it. The biggest problem caused by dynamic programming is the curse of dimensionality \cite{bellman1966dynamic}, which means in order to discretize on each dimension of the state space and obtain sufficient samples to approximate expectations, the amount of data required usually increases exponentially with the increase of dimensionality. 

\section{Adaptive dynamic programming}
\label{section_DP}
To overcome the curse of dimensionality in dynamic programming, Werbos first proposed the framework of adaptive dynamic programming (ADP) in 1977 \cite{werbos1977advanced}. The main idea is to use different function approximation structures to estimate the value function and policy. Instead of solving the problem point by point, we can solve it in time order, which greatly reduces the computational and storage requirements. In 1992, Werbos provided four basic algorithm structures of ADP: heuristic dynamic programming(HDP), dual heuristic dynamic programming (DHP), action-dependent HDP (ADHDP) and action-dependent DHP (ADDHP) \cite{werbos1992approximate}. 

We choose ADHDP, one of the most widely-used, to solve VWAP targeting execution problem, which is the first time ADP has been used in this problem as far as we know. Action-dependent (AD) means the action is a part of inputs to the critic network, and the algorithm consists of three main components: environment, action network, and critic network. The action network receives the current state and then provides the approximate optimal control, while the critic network is applied to approximate the optimal value function and gives feedback to the action network through the HJB equation. The complete algorithm is shown in Algorithm \ref{ADHDP}.

The action network is denoted by $ NN^{a} $ and the critic network is denoted by $ NN^{c} $, with the critic network being updated by the Monte Carlo method. Let's divide $ [0,T] $ into $ t_0, t_1,..., t_N $ and denote the time step by $ h $. Starting at $ t_0 $, we input the current state into the action network and get the action $ v_0 $ as feedback. Then after interacting with the environment, we gain the reward $ r_0 $ and the state $ \mathbf{Y}_{t_1} $, and input them into the action network again. The terminal reward $ f(\mathbf{Y}_T) $ is obtained until $ T $ by repeating the process. According to (\ref{value_function}), the approximate value at $ t_i $ given by the critic network $ NN^c(\mathbf{Y}_{t_i}) $ should minimize the following loss function: 
\begin{displaymath}
	e_i^c=\left|(r_{t_i}+r_{t_{i+1}}+...+r_{t_{N-1}})h+f(\mathbf{Y}_T)-NN^c(\mathbf{Y}_{t_i}) \right|^2. 
\end{displaymath}
Considering the terminal condition $ f(\mathbf{Y}_T) $, the loss function for updating the critic network is
\begin{equation}
	E^c=\sum\limits_{i=0}^{N-1}e_i^c+\left|f(\mathbf{Y}_T)-NN^c(\mathbf{Y}_{T})\right|^2. 
	\label{critic_loss_function}
\end{equation}

As for the action network, the optimal policy must satisfy the first order condition, which means the optimal trading speed satisfies the following relationship with the optimal value function according to (\ref{best_speed}):
\begin{displaymath}
	v^*=\dfrac{S\partial_xH-\partial_qH-b\partial_SH+2\varphi\rho\mu}{2(k\partial_xH+\varphi)}. 
\end{displaymath}
From the above equation, we only need the parameter estimation for $ b $ and $ k $, which is relatively easy since we know the likelihood functions. Meanwhile, we are able to skip the estimation of the parameters involved in the market trading speed. We denote estimation results by $ \hat{b} $ and $\hat{k}$, which will be introduced in Subsection \ref{param_estimation}. By applying the automatic differentiation technique of neural networks, the loss function to update the action network is
\begin{equation}
	E^a=\sum\limits_{i=0}^{N-1}\left|NN^a(\mathbf{Y}_{t_i})-\hat{v}_i^a\right|^2 \label{actor_loss_function},
\end{equation}
where $
\hat{v}_i^a=\dfrac{S_{t_i}\partial_xNN^c(\mathbf{Y}_{t_i})-\partial_qNN^c(\mathbf{Y}_{t_i})-\hat{b}\partial_SNN^c(\mathbf{Y}_{t_i})+2\varphi\rho\mu_{t_i}}{2(\hat{k}\partial_xNN^c(\mathbf{Y}_{t_i})+\varphi)} $. 

\begin{algorithm}
	\caption{Adaptive dynamic programming (ADP)}
	\label{ADHDP}
	\begin{algorithmic}[1]
		\renewcommand{\algorithmicrequire}{\textbf{Input}}
		\State {Environment $ Env $, initial state $ \mathbf{y}_0=(X_0,S_0,\mu_0,\mathfrak{N}) $, terminal time $ T $, time step $ h $, number of time
			grids $ G $, penalty parameter $\varphi$, number of episodes $ M $; }
		\State{Initialize action network $ NN^a $ and critic network $ NN^c $; }
		\For{$ m $=$ 1:M $}
		\For{$ i=0:G-1 $}
		\State {At time $ t_i=ih $, input current state $ \mathbf{Y}_{t_i} $ into $ NN^a $ to get action $ v_{t_i} $; }
		\State{Compute the value of current state $ NN^c(\mathbf{Y}_{t_i}) $; }
		\State{Compute the derivative of $ NN^c(\mathbf{Y}_{t_i}) $ with respect to $ \mathbf{Y}_{t_i} $; }
		\State{Take the action $v_{t_i}$, obtain new state $ \mathbf{Y}_{t_{i+1}} $ and receive the reward $ r_{t_i} $ from $ Env $; }
		\EndFor
		\State{At time $ T $, clear all inventory and obtain the terminal reward $ f(\mathbf{Y}_{T}) $; }
		\State{Update the critic network with the loss function $ E^c $ by (\ref{critic_loss_function}); }
		\State{Update the action network with the loss function $ E^a $ by (\ref{actor_loss_function}); }
		\EndFor
	\end{algorithmic}
\end{algorithm}
The above algorithm is also used to learn a stochastic policy, because we similarly derive the relationship between the optimal policy and the value function of the exploratory problem, which will be shown in Theorem \ref{thm2}, and we demonstrate the results in Section \ref{section_result}. The algorithm in the exploratory setting is hereafter referred to simply as exploratory ADP. 

\section{Exploratory liquidation problem}
\label{section_exploratory}
Referring to the framework of RL in continuous time and space created Jia and Zhou in \cite{jia2022policy,jia2022policy2}, we aim to identify a policy that incorporates randomness to encourage exploration. Let $ \mathcal{P}(\mathcal{A}) $ be a collection of probability density functions on the action space $ \mathcal{A}\subset\mathbb{R} $. Specifically, $ \pi=\{\pi_t(\cdot)=\pi(\cdot\mid\mathcal{F}_t^{W,N})\}_{0\leq t\leq T} $ is a policy, where $ \pi_t\in\mathcal{P}(\mathcal{A}) $ and $ v^\pi_t $ is the action determined by $ \pi_t $. Assuming the randomness is modeled by the random variable $ Z $, which is independent from $ W $ and $ N $, so we expand the original filtered probability space to $ (\Omega,\mathcal{F},\mathbb{P},\{\mathcal{F}_t\}_{0\leq t\leq T})$, where $\mathcal{F}_t=\mathcal{F}^{W,N}_t\vee\sigma(Z)$ and $ \mathbb{P} $ is a probability measure on $ \mathcal{F} $. Keeping the same notation as before, we denoted $ \mathbb{E}^{\mathbb{P}} $ as the expectation on $ (\Omega,\mathcal{F},\mathbb{P},\{\mathcal{F}_t\}_{0\leq t\leq T})$. 

\subsection{Problem formulation}
Given a policy $\pi$, let's replace the superscript $ v $ in the original problem with $\pi$ to represent the state $ \mathbf{Y}^\pi=\{\mathbf{Y}^\pi_t\}_{0\leq t\leq T} $, which follows
\begin{equation*}
	\d \mathbf{Y}_t^\pi=\mathbf{b}(t,\mathbf{Y}_t^\pi,v_t^\pi)\d t+\bm{\sigma} \d W_t+\bm{\eta}_{1+N_{t-}} \d N_t.
\end{equation*}
Based on deduction in \cite{wang2020reinforcement}, we get an exploratory version of the previous problem, and $ \mathbf{Y}^\pi $ has the same distribution as $\tilde{\mathbf{Y}}^\pi=\{\tilde{\mathbf{Y}}_t^\pi\}_{0\leq t\leq T}$, where $ \tilde{\mathbf{Y}}_t^\pi=(\tilde{S}_t^\pi,\tilde{X}_t^\pi,\tilde{Q}_t^\pi,\mu_t)^{\mathsf{T}}$ satisfies
\begin{equation}
	\begin{aligned}
		&\d \tilde{Q}^\pi_t=-\int_{\mathcal{A}}v\pi_t(v)\d v\d t,\\
		&\d \tilde{S}^\pi_t=-b\int_{\mathcal{A}}v\pi_t(v)\d v\d t+\sigma \d W_t,\\
		&\d \tilde{X}^\pi_t=\int_{\mathcal{A}}(S^\pi_t-kv)v\pi_t(v)\d v\d t, 
	\end{aligned}
	\label{ex_equation}
\end{equation}
and $\mu_t$ satisfies the same SDE (\ref{trading_speed}) due to the assumption that the market trading speed is not influenced by agents. In fact,  $ \mathbf{Y}^\pi $ are sample trajectories generated by policy $ \pi $, while $ \tilde{\mathbf{Y}}^\pi $ are the average over $\mathcal{A}$. We represent $ \tilde{\mathbf{Y}}_t^\pi $ in vector form:
\begin{equation}
	\d \tilde{\mathbf{Y}}_t^\pi=
	\begin{pmatrix}
		-b\int_{\mathcal{A}}v\pi_t(v)\d v\\
		\int_{\mathcal{A}}(\tilde{S}^\pi_t-kv)v\pi_t(v)\d v\\
		-\int_{\mathcal{A}}v\pi_t(v)\d v\\-\kappa\mu_t
	\end{pmatrix}\d t+\begin{pmatrix}
		\sigma\\0\\0\\0
	\end{pmatrix}\d W_t+\begin{pmatrix}
		0\\0\\0\\\eta_{1+N_{t-}}
	\end{pmatrix}\d N_t,
	\label{ex_equation_vec}
\end{equation}
which is abbreviated as
\begin{equation*}
	\d \tilde{\mathbf{Y}}_t^\pi=\tilde{\mathbf{b}}\left(t,\tilde{\mathbf{Y}}_t^\pi,\pi\left(\cdot\mid t,\tilde{\mathbf{Y}}_t^\pi\right)\right)\d t+\bm{\sigma} \d W_t+\bm{\eta}_{1+N_{t-}} \d N_t. 
\end{equation*}

We add a regularizer
into the value function to measure the exploration of the policy $ \pi $, which is 
\begin{equation*}
	\mathcal{H}(\pi):=-\int_{0}^{T}\int_{\mathcal{A}}\pi_t(v)\log\pi_t(v)\d v\d t. 
\end{equation*}
Let $ \gamma\geq0 $ represent the level of encouragement for exploration, which is referred to the temperature parameter in RL. Let the initial state be $ \mathbf{y}=(x,S,\mu,q)^{\mathsf{T}} $, so the optimization problem becomes
\begin{equation}
	\label{explore_optimal_problem}
	\begin{aligned}
		\sup\limits_{\pi\in\Pi}\mathbb{E}^{\mathbb{P}}_{t,\mathbf{y}}\bigg[f(\mathbf{Y}_T^\pi)-\int_{t}^{T}(\varphi(v^\pi_u-\mu_u)^2	+\gamma\log\pi_u(v^\pi_u))\pi_u(v^\pi_u)\d u\bigg], 
	\end{aligned}
\end{equation}
where $ \Pi $ is the set of all admissible policies and $ f(\mathbf{Y}_T^\pi)=X^\pi_T+Q^\pi_T(S^\pi_T-\alpha Q^\pi_T) $. Referring to \cite{jia2022policy2}, we give the precise definition of admissible policies. 

\begin{definition}
	A policy $ \pi $ is called admissible, if it satisfies:  
	\begin{enumerate}
		\item $ \pi\left(\cdot\mid t,\mathbf{y}\right) $ is a probability distribution on action space $\mathcal{A}$ and $ \pi\left(\cdot\mid t,\mathbf{y}\right)\in\mathcal{F}_t $; 
		\item For any initial $ \left(t,\mathbf{y}\right) \in [0,T]\times\mathbb{R}^4 $,  (\ref{ex_equation}) admits a unique weak solution; 
		\item $ \pi\left(\cdot\mid t,\mathbf{y}\right) $ is continuous in $ \left(t,\mathbf{y}\right) $, and uniformly Lipschitz continuous in $ \mathbf{y} $, i.e. for all $ \mathbf{y}, \mathbf{y}'\in\mathbb{R}^4 $, 
		\begin{align*} \lim\limits_{(t,\mathbf{y})\rightarrow(t',\mathbf{y}')}&\int_{\mathcal{A}}|\pi(v\mid t,\mathbf{y})-\pi(v\mid t',\mathbf{y}')|\d v=0\\
			\exists C>0,\qquad &\int_{\mathcal{A}}|\pi\left(v\mid t,\mathbf{y}\right)-\pi\left(v\mid t,\mathbf{y}'\right)|\d v\leq C|\mathbf{y}-\mathbf{y}'|. 
		\end{align*}
		\item $ \exists C>0,\ p\geq1 $, such that for all $ \left(t,\mathbf{y}\right)\in[0,T]\times\mathbb{R}^4 $, 
		\begin{equation*}
			\int_{\mathcal{A}}|\varphi(v-\mu)^2+\gamma\log\pi\left(v\mid t,\mathbf{y}\right)|\pi\left(v\mid t,\mathbf{y}\right)\d v\leq C(1+|\mathbf{y}|^p). 
		\end{equation*}
	\end{enumerate}
	\label{admi_def}
\end{definition}
Given an admissible policy $ \pi $, the related value function is defined as 
\begin{equation}
	\begin{aligned}
		V^\pi\left(t,\mathbf{y}\right)&=	\mathbb{E}^{\mathbb{P}}_{t,\mathbf{y}}\bigg[f(\mathbf{Y}_T^\pi)-\int_{t}^{T}(\varphi(v^\pi_u-\mu_u)^2+\gamma\log\pi_u(v^\pi_u))\pi_u(v^\pi_u)\d u\bigg]\\
		&=\mathbb{E}^{\mathbb{P}^{W,N}}_{t,\mathbf{y}}\bigg[f(\tilde{\mathbf{Y}}_T^\pi)-\int_{t}^{T}\int_{\mathcal{A}}(\varphi(v-\mu_u)^2+\gamma\log\pi_u(v))\pi_u(v)\d v\d u\bigg], 
	\end{aligned}\label{value_V}
\end{equation}
in which the second equation holds because $ \mathbf{Y}^\pi $ and $ \tilde{\mathbf{Y}}^\pi  $ have the same distribution, and the optimal value function is denoted by $ V\left(t,\mathbf{y}\right)=\sup\limits_{\pi\in\Pi}	V^\pi\left(t,\mathbf{y}\right).$
The following theorem guarantees the well-posedness of our exploratory problem. 
\begin{theorem}
	\label{thm:growth_control}
	If Assumption \ref{assumption1} holds, given an admissible policy $\pi$,  (\ref{ex_equation_vec}) admits a unique strong solution. Furthermore, if $p\geq2$, then there exists $ C=C(p) $ satisfying \begin{equation}
		\mathbb{E}^{\mathbb{P}^{W,N}}\left.\left[\max\limits_{t\leq s\leq T}\left|\tilde{\mathbf{Y}}_s^\pi\right|^p\right|\tilde{\mathbf{Y}}_t^\pi=\mathbf{y}\right]\leq C(1+\left|\mathbf{y}\right|^p). 
		\label{grow_condition}
	\end{equation}
	And the value function (\ref{value_V}) is finite. 
\end{theorem}
\begin{proof}
	First, according to the definition of $ \tilde{\mathbf{b}} $, we have
	\begin{equation*}\resizebox{.99\hsize}{!}{
			$ \tilde{\mathbf{b}}(t,\mathbf{y},\pi(\cdot\mid t,\mathbf{y}))-\tilde{\mathbf{b}}(t,\mathbf{y}',\pi(\cdot\mid t,\mathbf{y}'))=\int_{\mathcal{A}}\left[\pi(\cdot\mid t,\mathbf{y})\mathbf{b}(t,\mathbf{y},v)-\pi(\cdot\mid t,\mathbf{y}')\mathbf{b}(t,\mathbf{y}',v)\right]\d v.  $}
	\end{equation*}
	Due to the linear growth property and uniform Lipschitz continuity of $ \mathbf{b} $, we get the following estimation, 
	\begin{align*}
		&|\tilde{\mathbf{b}}(t,\mathbf{y},\pi(\cdot \mid t,\mathbf{y}))-\tilde{\mathbf{b}}(t,\mathbf{y}',\pi(\cdot\mid t,\mathbf{y}'))|\\
		\leq&\int_{\mathcal{A}}|\pi(\cdot\mid t,\mathbf{y})-\pi(\cdot\mid t,\mathbf{y}')||\mathbf{b}(t,\mathbf{y},v)-\mathbf{b}(t,\mathbf{0},v)|\d v\\
		&+\int_{\mathcal{A}}|\mathbf{b}(t,\mathbf{0},v)||\pi(\cdot\mid t,\mathbf{y})-\pi(\cdot\mid t,\mathbf{y}')|\d v+\int_{\mathcal{A}}|\mathbf{b}(t,\mathbf{y},v)-\mathbf{b}(t,\mathbf{y}',v)|\pi(\cdot\mid t,\mathbf{y}')\d v\\
		\leq& (C_1|\mathbf{y}|+C_2)\int_{\mathcal{A}}|\pi(\cdot\mid t,\mathbf{y})-\pi(\cdot\mid t,\mathbf{y}')|\d v+C_1|\mathbf{y}-\mathbf{y}'|\\
		\leq& C_1|\mathbf{y}-\mathbf{y}'|+(C_1|\mathbf{y}|+C_2)C_3|\mathbf{y}-\mathbf{y}'|. 
	\end{align*}
	That is to say $\tilde{\mathbf{b}}$ is locally Lipschitz continuous. The standard SDE theory guarantees that the assumption of the Lipschitz condition in Assumption \ref{assumption1} can be relaxed to the local Lipschitz condition. Using the linear growth property of $\mathbf{b}$ again, we have \begin{align*}
		|\tilde{\mathbf{b}}(t,\mathbf{y},\pi(\cdot \mid t,\mathbf{y}))|\leq \int_{\mathcal{A}}(C_1|\mathbf{y}|+C_2)\pi(v|t,\mathbf{y})\d v=C_1|\mathbf{y}|+C_2, 
	\end{align*}
	so (\ref{ex_equation_vec}) also admits a unique strong solution. Reform (\ref{ex_equation_vec}) in integral form \begin{equation*}
		\tilde{\mathbf{Y}}^\pi_s=\mathbf{y}+\int_{t}^{s}\tilde{\mathbf{b}}\left(u,\tilde{\mathbf{Y}}_u^\pi,\pi_u\right)\d u+\bm{\sigma} (W_s-W_t)+\int_t^s\bm{\eta}_{1+N_u^-}\d N_u,
	\end{equation*}
	and then we obtain the following estimation,  \begin{align*}
		&\mathbb{E}^{\mathbb{P}^{W,N}}\left.\left[\max\limits_{t\leq s\leq T}\left|\tilde{\mathbf{Y}}_s^\pi\right|^p\right|\tilde{\mathbf{Y}}_t^\pi=\mathbf{y}\right]\\
		\leq& C_1\mathbb{E}^{\mathbb{P}^{W,N}}\left[|\mathbf{y}|^p+\max\limits_{t\leq s\leq T}\left|\int_{t}^s\tilde{\mathbf{b}}(u,\tilde{\mathbf{Y}}_u^\pi,\pi(\cdot\mid u,\tilde{\mathbf{Y}}^\pi_u))\d u\right|^p\right.\\
		&\qquad\qquad+\sigma\max\limits_{t\leq s\leq T}|W_s-W_t|^p+\left|\sum_{m=N_t+1}^{N_T}\bm{\eta}_m\right|^p
		\bigg|\tilde{\mathbf{Y}}_t^\pi=\mathbf{y}\bigg]\\
		\leq& C_1\mathbb{E}^{\mathbb{P}^{W,N}}\left[\left.|\mathbf{y}|^p+C_2\int_{t}^{T}\left(1+\max\limits_{t\leq s\leq u}\left|\tilde{\mathbf{Y}}_s^\pi\right|\right)^p \d u\right|\tilde{\mathbf{Y}}_t^\pi=\mathbf{y}\right]+C_3\\
		\leq& C_4(1+|\mathbf{y}|^p)+C_5\int_{t}^{T}\mathbb{E}^{\mathbb{P}^{W,N}}\left[\left.\max\limits_{t\leq s\leq u}\left|\tilde{\mathbf{Y}}_s^\pi\right|^p\right|\tilde{\mathbf{Y}}_t^\pi=\mathbf{y}\right]\d u, 
	\end{align*}
	where the second inequality holds because of the linear growth of $\tilde{\mathbf{b}}$, $ W_s-W_t\sim\mathcal{N}(0,s-t) $ and the finite moment of $ \bm{\eta}_m $ for all $ m\in\mathbb{Z}^+ $. 
	Then we get (\ref{grow_condition}) by applying the Gronwall's inequality, and combine this result with Definition \ref{admi_def} to know that $ V^\pi(t,\mathbf{y}) $ is finite. 
\end{proof}
\subsection{The expression of the optimal policy}
Using Bellman's principle of optimality and standard methods for stochastic control problems, we know that the value function $ V $ satisfies the following HJB equation
\begin{equation}
	\label{HJB_V}
	\begin{aligned}
		0=\bigg(\partial_t+\frac{1}{2}\sigma^2\partial_{SS}&+\mathcal{L}^\mu\bigg)V+\sup\limits_{\pi\in\Pi}\bigg\{
		\int_{\mathcal{A}}[(S-kv)v\partial_xV\\
		&-v\partial_qV-bv\partial_SV-\gamma\log\pi_t(v)-\varphi(v-\rho\mu)^2]\pi_t(v)\d v\bigg\},
	\end{aligned}
\end{equation}
with the terminal condition
\begin{equation}
	V(T,x,S,\mu,q)=x+q(S-\alpha q). 
	\label{terminal_V}
\end{equation}
The following theorem provides the formula of the optimal control for the HJB (\ref{HJB_V}).
\begin{theorem}
	\label{thm2}
	The policy satisfying (\ref{HJB_V}) and the terminal condition (\ref{terminal_V}) is \begin{equation}
		\pi^*(v\mid t,\mathbf{y})=\mathcal{N}\left(\dfrac{-w_1(t,\mu)-2w_2(t,\mu)\mu-bq+2\varphi\rho\mu}{2(\varphi+k)},\dfrac{\gamma}{2(\varphi+k)}\right),
		\label{pi*_normal}
	\end{equation}
	which is a Gaussian distribution with the mean value $ \dfrac{-w_1(t,\mu)-2w_2(t,\mu)\mu-bq+2\varphi\rho\mu}{2(\varphi+k)} $ and the variance $ \dfrac{\gamma}{2(\varphi+k)} $. Here are the expressions of parameters,  
	\begin{equation*}
		\begin{aligned}
			&w_2(t,\mu)=-\left(\dfrac{T-t}{k+\varphi}+\dfrac{1}{\alpha-\frac{1}{2}b}\right)^{-1}-\dfrac{1}{2}b,\quad w_1(t,\mu)=l_0(t)+l_1(t)\mu,\\
			&l_1(t)=2\varphi\rho\left((T-t)+\frac{k+\varphi}{\alpha-\frac{1}{2}b}\right)^{-1}\dfrac{1-e^{-\kappa(T-t)}}{\kappa},\\
			&l_0(t)=2\varphi\rho\lambda\mathbb{E}[\eta]\left((T-t)+\frac{k+\varphi}{\alpha-\frac{1}{2}b}\right)^{-1}\dfrac{e^{-\kappa(T-t)}-1+\kappa(T-t)}{\kappa^2}. 
		\end{aligned}
	\end{equation*}
\end{theorem}
\begin{proof}
	Consider the optimization problem in (\ref{HJB_V}): 
	\begin{equation}
		\resizebox{.9\hsize}{!}{$\begin{aligned}
				\sup\limits_{\pi\in\Pi}\left\{
				\int_{\mathcal{A}}[(S-kv)v\partial_xV-v\partial_qV-bv\partial_SV-\gamma\log\pi_t(v)-\varphi(v-\rho\mu)^2]\pi_t(v)\d v\right\}.
			\end{aligned} 
			$} \label{optimal_int}
	\end{equation}
	We define the function $ f(z):=[(S-kv)v\partial_xV-v\partial_qV-bv\partial_SV-\gamma\log z-\varphi(v-\rho\mu)^2]z $
	and the unique point $ z^* $ to attain the maximum of $ f(z) $ is 
	\begin{equation*}
		z^*=\exp\left(\frac{1}{\gamma}[(S-kv)v\partial_xV-v\partial_qV-bv\partial_SV-\varphi(v-\rho\mu)^2-\gamma]\right). 
	\end{equation*}
	Therefore, the probability density function maximizing (\ref {optimal_int}) satisfies
	\begin{equation}
		\pi^*(v\mid t,\mathbf{y})\propto \exp\left(\frac{1}{\gamma}[(S-kv)v\partial_xV-v\partial_qV-bv\partial_SV-\varphi(v-\rho\mu)^2-\gamma]\right).  
		\label{exp_pi}
	\end{equation}
	After arranging the exponential terms, we know that $ \pi^* $ is a Gaussian distribution,  \begin{equation*}
		\pi^*(v\mid t,\mathbf{y})=\mathcal{N}\left(
		\dfrac{-\partial_qV-b\partial_SV+2\varphi\rho\mu+S\partial_xV}{2(\varphi+k\partial_x V)}, \dfrac{\gamma}{2(\varphi+k\partial_x V)}
		\right). 
	\end{equation*}
	Substitute the formula of $\pi^*(v\mid t,\mathbf{y})$ to (\ref{HJB_V}) and rearrange, then we get 
	\begin{equation*}
		\begin{aligned}
			\left(\partial_t+\frac{1}{2}\sigma^2\partial_{SS}+\mathcal{L}^\mu\right)V&+\dfrac{(\partial_qV+b\partial_SV-2\varphi\rho\mu-S\partial_xV)^2}{4(\varphi+k\partial_xV)}\\
			&+\gamma\log\sqrt{\dfrac{\gamma\pi}{\varphi+k\partial_xV}}-\varphi\rho\mu^2=0. 
		\end{aligned}
	\end{equation*}
	
	First make the ansatz $ V(t,x,S,\mu,q)=x+qS+w(t,\mu,q)  $ as in \cite{cartea2016closed}, 
	and $ w $ satisfies the terminal condition $  w(T,\mu,q)=-\alpha q^2$. 
	Then we get the equation of $ w(t
	,\mu,q) $,
	\begin{equation*}
		\partial_tw+\mathcal{L}^\mu w+\dfrac{(\partial_qw+bq-2\varphi\rho\mu)^2}{4(\varphi+k)}+\gamma\sqrt{\dfrac{\gamma\pi}{\varphi+k}}-\varphi\rho\mu^2=0. 
	\end{equation*}
	Similarly, make an ansatz of $ w(t,\mu,q)=w_0(t,\mu)+w_1(t,\mu)q+w_2(t,\mu)q^2 $, in which $ w_0,\ w_1$ and $ w_2 $ satisfy the coupling equations:
	\begin{subequations}
		\begin{align}
			&0=(\partial_t+\mathcal{L}^\mu)w_2+\dfrac{\left(w_2+\frac{1}{2}b\right)}{k+\varphi},\label{eq:w_2}\\
			&0=(\partial_t+\mathcal{L}^\mu)w_1+\dfrac{w_1-2\varphi\rho\mu}{k+\varphi}\left(w_2+\frac{1}{2}b\right),\label{eq:w_1}\\
			&0=(\partial_t+\mathcal{L}^\mu)w_0+\dfrac{1}{4(k+\varphi)}(w_1-2\varphi\rho\mu)^2+\gamma\sqrt{\dfrac{\gamma\pi}{\varphi+k}}-\varphi\rho^2\mu^2\label{eq:w_0},
		\end{align}
	\end{subequations}
	with terminal conditions $ w_2(T,\mu)=-\alpha$ and $\ w_1(T,\mu)=w_0(T,\mu)=0 $. Notice that there is no term about $\mu$ in (\ref{eq:w_2}) and its terminal condition also does not depend on $\mu$, so we know $ \mathcal{L}^\mu w_2\equiv0 $. Thus, we can solve $ w_2 $ from (\ref{eq:w_2}): 
	\begin{equation*}
		w_2(t,\mu)=-\left(\dfrac{T-t}{k+\varphi}+\dfrac{1}{\alpha-\frac{1}{2}b}\right)^{-1}-\dfrac{1}{2}b. 
	\end{equation*}
	Observing (\ref{eq:w_1}), it is reasonable to assume $ w_1 $ is linear in $\mu$. Hence we make the ansatz $ w_1(t,\mu)=l_0(t)+l_1(t)\mu $, and the terminal condition implies that $ l_0(T)=l_1(T)=0 $. Then (\ref{eq:w_1}) becomes
	\begin{equation*}
		0=\left\{\partial_tl_0+\dfrac{w_2+\frac{1}{2}b}{k+\varphi}l_0+\lambda\mathbb{E}[\eta] l_1\right\}+\left\{\partial_t l_1-\kappa l_1+\dfrac{l_1-2\varphi\rho}{k+\varphi}\left(h_2+\frac{1}{2}b\right)\right\}\mu. 
	\end{equation*}
	The terms in two braces must be zero, since the above equation holds true for all $\mu$. We first solve $ l_1 $ and then solve $ l_0 $ in the similar way, so we have
	\begin{equation*}
		\begin{aligned}
			l_1(t)&=2\varphi\rho\left((T-t)+\frac{k+\varphi}{\alpha-\frac{1}{2}b}\right)^{-1}\dfrac{1-e^{-\kappa(T-t)}}{\kappa},\\
			l_0(t)&=2\varphi\rho\lambda\mathbb{E}[\eta]\left((T-t)+\frac{k+\varphi}{\alpha-\frac{1}{2}b}\right)^{-1}\dfrac{e^{-\kappa(T-t)}-1+\kappa(T-t)}{\kappa^2}. 
		\end{aligned}
	\end{equation*}
	There is no need for us to solve (\ref{eq:w_0}) since $\pi^*$ does not depend on $ w_0 $. Finally notice that the mean value can be represented as
	\begin{equation*}
		\dfrac{-\partial_qV-b\partial_SV+2\varphi\rho\mu+S\partial_xV}{2(\varphi+k\partial_x V)}=\dfrac{-w_1(t,\mu)-2w_2(t,\mu)\mu-bq+2\varphi\rho\mu}{2(\varphi+k)}, 
	\end{equation*}
	and we get the conclusions. 
\end{proof}
Substitute the expression of $ w_2 $ to the mean value of $ \pi^*(v|t,\mathbf{y}) $, and we have
\begin{equation*}
	\dfrac{-w_1-2w_2\mu-bq+2\varphi\rho\mu}{2(\varphi+k)}=\dfrac{1}{\varphi+k}\left[\varphi\rho\mu-\dfrac{1}{2}w_1(t,\mu)\right]+\dfrac{q}{(T-t)+\zeta},
	\label{almost_twap}
\end{equation*}
where $\zeta=(k+\varphi)/(\alpha-\frac{1}{2}b)$. This is a TWAP-like liquidation policy, since the second term on the 
right-hand side is similar to the TWAP policy. 
The subsequent result elucidates the relationship between the exploratory optimal policy and the one presented in (\ref{best_speed_h}). More complete analysis of exploratory HJB equations and convergence analysis as the level of exploration decays to zero can be find in \cite{wang2018exploration}.
\begin{proposition}
	For any $ (t,\mathbf{y})\in[0,T]\times\mathbb{R}^4 $, we have
	\begin{equation*}
		\lim\limits_{\gamma\rightarrow 0}\pi^*(\cdot\mid t,\mathbf{y})\Rightarrow \delta_{v^*}(\cdot),
	\end{equation*}
	where $\Rightarrow$ represents weak convergence of probability measures and $ \delta_{v^*}(\cdot) $ is the Dirac measure centered on $ v^*$, which is shown in (\ref{best_speed}). 
\end{proposition}
\begin{proof}
	It is direct to get the convergence from the explicit expression (\ref{pi*_normal}). 
\end{proof}

\section{Reinforcement learning algorithm}
\label{section_RL} 
In this section, we extend the framework developed by Jia and Zhou \cite{jia2022policy,jia2022policy2} to stochastic processes with Poisson jumps, and provide some theoretical proofs for policy evaluation and policy improvement. 

To find the optimal policy through learning algorithms, first we need to parameterize the policy $ \pi $ and its value function $ V $. Denote the parameterized policy by $ \pi^\Phi $ and the parameterized value function by $ V^{\Theta} $, where $\Phi$ and $\Theta$ stand for the corresponding parameters. $ \pi^\Phi $ and $ V^{\Theta} $ are respectively known as the action network and the critic network. Notice that $ V^{\Theta} $ is an approximating network of $ V^{\pi^\Phi} $. We first make the following regularity assumption for the critic network.  
\begin{assumption}
	For all $\Theta$, $ V^\Theta(t,\mathbf{y}) $ is sufficiently smooth, i.e. all the derivatives required exist in the classical sense. Moreover, $ V^\Theta(t,\mathbf{y}) $ and $ \frac{\partial V^\Theta}{\partial \Theta}(t,\mathbf{y}) $ are the polynomial growth in $ \mathbf{y} $. 
	\label{assumption:parameterized_V}
\end{assumption}

\subsection{Policy evaluation}
Given an admissible policy $\pi$, our first goal is to find its corresponding value function $ V^\pi(t,\mathbf{y}) $. According to Feynman-Kac theorem, $ V^\pi(t,\mathbf{y}) $ can be characterized by PDE. That is to say, if $ g(t,\mathbf{y})\in C^{1,2}([0,T]\times\mathbb{R}^4) $ such that $ \frac{\partial g}{\partial \mathbf{y}} $, $ \frac{\partial^2 g}{\partial \mathbf{y}^2}  $ are polynomial growth, which also satisfies
\begin{align}\label{feymann}
	&\int_{\mathcal{A}} \pi_t(v)\left[\mathcal{L}^{\pi(v)}g(t,\mathbf{y})-\varphi(v-\rho\mu)^2-\gamma\log(\pi_t(v))\right]\d v=0
\end{align}
with terminal conditions $ g(T,\mathbf{y})=x+qS-\alpha q^2,\  \forall \mathbf{y}\in\mathbb{R}^4 $, then $ g(t,\mathbf{y})$ is the value function $V^\pi(t,\mathbf{y}) $. In (\ref{feymann}), $ \mathcal{L}^{\pi(v)}g(t,\mathbf{y})=\dfrac{\partial g}{\partial t}+\dfrac{\partial g}{\partial X}(S-kv)v-bv\dfrac{\partial g}{\partial S}+\dfrac{1}{2}\sigma^2\dfrac{\partial^2g}{\partial S^2}+\mathcal{L}^\mu g-v\dfrac{\partial g}{\partial q} $, represents an infinitesimal generator given the policy $\pi$. 

We denote the natural filtration generated by $ \{\tilde{\mathbf{Y}}^\pi_t\}_{0\leq t\leq T} $ as $ \mathcal{F}^{\tilde{\mathbf{Y}}^\pi}=\{\mathcal{F}^{\tilde{\mathbf{Y}}^\pi}_t\}_{0\leq t\leq T} $, and define $ \mathcal{F}^{{\mathbf{Y}}^\pi}=\{\mathcal{F}^{{\mathbf{Y}}^\pi}_t\}_{0\leq t\leq T} $ in the same way. The following theorem provides a martingale condition, which is the theoretical core of policy evaluation.
\begin{theorem}
	\label{thm:martingale}
	Let $ M_s=V^\pi\left(s,\tilde{\mathbf{Y}}_s^\pi\right)+\int_{t}^s\int_{\mathcal{A}}\pi_u(v)\left[r\left(\tilde{\mathbf{Y}}_u^\pi,v\right)-\gamma\log(\pi_u(v))\right]\d v\d u $, where $ \tilde{\mathbf{Y}}_s^\pi $ is the solution of (\ref{ex_equation_vec}) and $ r(\tilde{\mathbf{Y}}_u^\pi,v)=-\varphi(v-\rho\mu_u)^2 $ is the running reward function. Given the initial state $ (t,\mathbf{y})\in [0,T]\times\mathbb{R}^4 $, $ V^\pi(t,\mathbf{y}) $ is the value function of $\pi$ if and only if it satisfies terminal condition $ V^\pi(T,\mathbf{y})=f(\mathbf{y}) $, and $ \{M_s\}_{t\leq s\leq T} $
	is a $ \left(\mathcal{F}^{\tilde{\mathbf{Y}}^\pi},\mathbb{P}^{W,N}\right) $-martingale on $ [t,T]$. 
\end{theorem}
\begin{proof}
	If $ V^\pi(t,\mathbf{y}) $ is the value function of $\pi$, from (\ref{value_V}) we know that
	\begin{equation*}
		\begin{aligned}
			M_s&=V^\pi\left(s,\tilde{\mathbf{Y}}_s^\pi\right)+\int_{t}^s\int_{\mathcal{A}}\pi_u(v)\left[r\left(\tilde{\mathbf{Y}}_u^\pi,v\right)-\gamma\log(\pi_u(v))\right]\d v\d u	\\
			&=\mathbb{E}^{\mathbb{P}^{W,N}}\left[f\left(\tilde{\mathbf{Y}}_T^\pi\right)+\int_{t}^{T}\int_{\mathcal{A}}\pi_u(v)\left(r\left(\tilde{\mathbf{Y}}_u^\pi,v\right)-\gamma\log(\pi_u(v))\right)\d v\d u\bigg| \mathcal{F}_s^{\tilde{\mathbf{Y}}^{\pi}}\right]\\
			&=\mathbb{E}^{\mathbb{P}^{W,N}}\left[M_T\left.\bigg|\mathcal{F}_s^{\tilde{\mathbf{Y}}^{\pi}}\right]\right.,   
		\end{aligned}
	\end{equation*} 
	so $ \{M_s\}_{t\leq s\leq T} $ is a $ (\mathcal{F}^{\tilde{\mathbf{Y}}^\pi},\mathbb{P}^{W,N}) $-martingale because of the Markov property of $ \{\tilde{\mathbf{Y}}^\pi_s\}_{t\leq s\leq T} $.  
	
	On the other hand, if $ \tilde{V}^\pi $ is another function that satisfies the conditions in Theorem \ref{thm:martingale}, and
	$$ \tilde{M}_s=\tilde{V}^\pi\left(s,\tilde{\mathbf{Y}}_s^\pi\right)+\int_{t}^s\int_{\mathcal{A}}\pi_u(v)\left[r\left(\tilde{\mathbf{Y}}_u^\pi,v\right)-\gamma\log(\pi_u(v))\right]\d v\d u	 $$
	is also a martingale. Noticing that $ \tilde{V}^\pi(T,\tilde{\mathbf{Y}}^\pi_T)=f(\tilde{\mathbf{Y}}^\pi_T) $, we have
	
	\begin{align*}
		\tilde{M}_s & =\mathbb{E}^{\mathbb{P}^{W,N}}\left[\tilde{M}_T\left|\mathcal{F}_s^{\tilde{\mathbf{Y}}^{\pi}}\right]\right.                                                                                                                                                    \\
		& =\mathbb{E}^{\mathbb{P}^{W,N}}\left.\left[\tilde{V}^\pi(T,\tilde{\mathbf{Y}}_T^\pi)+\int_{t}^{T}\int_{\mathcal{A}}\pi_u(v)\left[r\left(\tilde{\mathbf{Y}}_u^\pi,v\right)-\gamma\log(\pi_u(v))\right]\d v\d u\right|\mathcal{F}_s^{\tilde{\mathbf{Y}}^{\pi}}\right] \\
		& =\mathbb{E}^{\mathbb{P}^{W,N}}\left.\left[f(\tilde{\mathbf{Y}}_T^\pi)+\int_{t}^{T}\int_{\mathcal{A}}\pi_u(v)\left[r\left(\tilde{\mathbf{Y}}_u^\pi,v\right)-\gamma\log(\pi_u(v))\right]\d v\d u\right|\mathcal{F}_s^{\tilde{\mathbf{Y}}^{\pi}}\right]               \\
		& =M_s,
	\end{align*}
	so we get the desired result. 
\end{proof}

Our task is to make sure $ V^{\Theta} $ is close enough to the current value function $ V^{\pi^\Phi} $. According to Jia and Zhou's framework \cite{jia2022policy2}, there are two methods to achieve it, and we theoretically prove that they can also be extended to processes with Poisson jumps.

\subsubsection{Martingale loss function}
According to Theorem \ref{thm:martingale}, we define a martingale $ \{M_s\}_{0\leq s\leq T} $ from the value function $ V ^\pi$. Applying martingale property and the same distribution of $ \mathbf{Y}^{\pi} $ and $ \tilde{\mathbf{Y}}^{\pi} $, we have
\begin{equation*}
	M_s=\mathbb{E}^{\mathbb{P}^{W,N}}\left[M_T\left|\tilde{\mathbf{Y}}^{\pi}_s=\mathbf{y}\right]\right.=\mathbb{E}^{\mathbb{P}}\left[M_T\left|{\mathbf{Y}}_s^{\pi}=\mathbf{y}
	\right]\right.,\qquad \forall s\in[0,T], 
\end{equation*}
which means
\begin{equation*}
	M_s=\arg\min\limits_{\xi\in\mathcal{F}_s^{\mathbf{Y}^{\pi}}}\mathbb{E}^{\mathbb{P}}|M_T-\xi|^2. 
\end{equation*}
When the current policy is $ \pi^\Phi $, we define $ M^\Theta_t $ from $ V^\Theta(t,\mathbf{Y}^{\pi^\Phi}_t) $ as above, and the martingale loss function is defined as
\begin{equation*}
	\begin{aligned}
		\text{ML}(\Theta)&=\mathbb{E}^{\mathbb{P}}\left[\int_{0}^{T}\left|M_T-M^{\Theta}_t\right|^2\d t\right]=\mathbb{E}^{\mathbb{P}}\bigg[\int_{0}^{T}\bigg(f(\mathbf{Y}^{\pi^\Phi}_T)-V^{\Theta}(t,\mathbf{Y}^{\pi^\Phi}_t)\\
		&\quad\quad\quad\quad\quad\quad\quad\quad\quad\ +\int_{t}^{T}\left[r(\mathbf{Y}_s^{\pi^\Phi},v_s^{\pi^\Phi})-\gamma\log(\pi^\Phi(v^{\pi^\Phi}_{s}\mid s,\mathbf{Y}^{\pi^\Phi}_{s}))\right]\d s\bigg)^2\d t\bigg]. 
	\end{aligned}
\end{equation*}
To find parameters minimizing the martingale loss function corresponds to the Monte Carlo algorithm. We use the random gradient descent method to update parameters, which is shown as below: 
\begin{equation}
	\Theta\leftarrow\Theta+\alpha_\Theta^{(1)}\Delta \Theta, 
\end{equation}
\begin{equation}
	\label{eq:ML_update}
	\begin{aligned}
		\Delta \Theta=\sum\limits_{i=0}^{N-1}\bigg(&f(\mathbf{Y}_{t_N}^{\pi^\Phi})+\sum\limits_{j=i}^{N-1}\bigg[r(\mathbf{Y}^{\pi^\Phi}_{t_j},v^{\pi^\Phi}_{t_j})-\gamma\log(\pi^\Phi(v^{\pi^\Phi}_{t_j}\mid t_j,\mathbf{Y}^{\pi^\Phi}_{t_j}))\bigg]h\\
		&-V^{\Theta}(t_i,\mathbf{Y}^{\pi^\Phi}_{t_i})\bigg)
		\dfrac{\partial V^{\Theta}}{\partial \Theta}(t_i,\mathbf{Y}^{\pi^\Phi}_{t_i})h,
	\end{aligned}
\end{equation}
where $ \alpha_\Theta^{(1)} $ is a positive constant representing the learning rate. 
This is an offline algorithm since we can't obtain $ \mathbf{Y}_{t_N}^{\pi^\Phi} $ until the episode ends. 
\subsubsection{Martingale orthogonality condition}
There is another way to use Theorem \ref{thm:martingale} to obtain the optimal parameter $\Theta$. The following proposition shows a necessary and sufficient condition for $ M^\Theta $ to be a martingale, which is called the martingale orthogonality condition. First, we introduce some notation. For any semi-martingale $ D $ and any filtration $ \mathcal{G}=\{\mathcal{G}_t\}_{0\leq t\leq T} $, we denote 
\begin{equation*}
	L^2_\mathcal{G}([0,T];D)=\left\{g=\{g_t\}_{0\leq t\leq T};\ g\ \text{is}\  \mathcal{G}_t\text{-progressively measurable and}  ||g||_2<\infty\right\}, 
\end{equation*}
where $ ||g||_2=(\mathbb{E}[\int_{0}^{T}g^2_t\d \langle D\rangle_t])^{\frac{1}{2}} $ is $ L^2- $norm defined in this space.  
\begin{proposition}
	\label{prop:martingale_orthogonality}
	Let $ M^\Theta_s=V^\Theta(s,\tilde{\mathbf{Y}}_s^\pi)+\int_{t}^s\int_{\mathcal{A}}\pi_u(v)[r(\tilde{\mathbf{Y}}_u^\pi,v)-\gamma\log(\pi_u(v))]\d v\d u. $
	$ M^\Theta=\{M^\Theta_t\}_{0\leq t\leq T} $ is a martingale if and only the following martingale orthogonality condition holds, i.e. for all $ \xi=\{\xi_t\}_{0\leq t\leq T}\in L^2_{\mathcal{F}^{\mathbf{Y}^\pi}}\left([0,T];V^\Theta\left(\cdot,\mathbf{Y}^\pi_\cdot\right)\right) $, 
	\begin{equation}
		\mathbb{E}^\mathbb{P}\int_{0}^{T}\xi_t\left[\d V^\Theta(t,\mathbf{Y}_t^\pi)+r(\mathbf{Y}_t^\pi,v^\pi_t)\d t-\gamma\log(\pi_t(v^\pi_t))\d t\right]=0,  
		\label{eq:martingale_orthogonality}
	\end{equation}
	where $ \xi $ is called the test function. 
	
\end{proposition}
\begin{proof}[Proof]
	First we prove that there is a one-to-one correspondence between these two spaces:  $L^2_{\mathcal{F}^{\tilde{\mathbf{Y}}^\pi}}([0,T];V^\Theta(\cdot,\tilde{\mathbf{Y}}^\pi_\cdot)) $ and $ L^2_{\mathcal{F}^{{\mathbf{Y}}^\pi}}\left([0,T];V^\Theta\left(\cdot,{\mathbf{Y}}^\pi_\cdot\right)\right) $. 
	For any $\xi\in L^2_{\mathcal{F}^{\tilde{\mathbf{Y}}^\pi}}([0,T];V^\Theta(\cdot,\tilde{\mathbf{Y}}^\pi_\cdot)) $, there exists a measurable function $ \bm{\xi}: [0,T]\times C([0,T];\mathbb{R}^4)\rightarrow \mathbb{R} $, and $\bm{\xi} (t,\tilde{\mathbf{Y}}^\pi_{t\wedge\cdot})=\xi_t $. 
	Let $\xi':= \bm{\xi}(t,\mathbf{Y}^\pi_{t\wedge\cdot}) $, which belongs to $ L^2_{\mathcal{F}^{\mathbf{Y}^\pi}}([0,T];V^\Theta(\cdot,\mathbf{Y}^\pi_\cdot)) $ because $ \tilde{\mathbf{Y}}^\pi $ and $ \mathbf{Y}^\pi $ have the same distribution. The correspondence on other side can be proved in the same way.  
	
	To prove (\ref{eq:martingale_orthogonality}), we notice that \begin{equation*}
		\d M_t^\Theta=\d V^\Theta\left(t,\tilde{\mathbf{Y}}_t^\pi\right)+\int_{\mathcal{A}}\pi_t(v)\left[r\left(\tilde{\mathbf{Y}}_t^\pi,v\right)-\gamma\log(\pi_t(v))\right]\d v\d t,  
	\end{equation*} 	
	so we know for all $ \xi\in L^2_{\mathcal{F}^{{\mathbf{Y}}^\pi}}\left([0,T];V^\Theta\left(\cdot,{\mathbf{Y}}^\pi_\cdot\right)\right) $, 
	\begin{align*}
		&\mathbb{E}^\mathbb{P}\int_{0}^{T}{\xi}_t\left[dV^\Theta(t,\mathbf{Y}_t^\pi)+r(\mathbf{Y}_t^\pi,v^\pi_t)dt-\gamma\log(\pi_t(v^\pi_t))\d t\right]\\
		=&\mathbb{E}^\mathbb{P}\int_{0}^{T}\bm{\xi}(t,{\mathbf{Y}}^\pi_{t\wedge\cdot})\left[dV^\Theta(t,\mathbf{Y}_t^\pi)+r(\mathbf{Y}_t^\pi,v^\pi_t)dt-\gamma\log(\pi_t(v^\pi_t))\d t\right]\\
		=&\mathbb{E}^{\mathbb{P}^{W,N}}\int_{0}^{T}\bm{\xi}(t,\tilde{\mathbf{Y}}^\pi_{t\wedge\cdot})\left\{dV^\Theta\left(t,\tilde{\mathbf{Y}}_t^\pi\right)+\int_{\mathcal{A}}\pi_t(v)\left[r\left(\tilde{\mathbf{Y}}_t^\pi,v\right)-\gamma\log(\pi_t(v))\right]\d v\d t\right\}\\
		=&\mathbb{E}^{\mathbb{P}^{W,N}}\int_{0}^{T}\xi'_t\d M_t^\Theta,\qquad \xi'\in L^2_{\mathcal{F}^{\tilde{\mathbf{Y}}^\pi}}([0,T];V^\Theta(\cdot,\tilde{\mathbf{Y}}^\pi_\cdot)),
	\end{align*}	
	where the first equation holds because of the correspondence proved above, and the same distribution of $ {\mathbf{Y}}^\pi $ and $ \tilde{\mathbf{Y}}^\pi $ makes the second equation hold true. 
	If $ M^\Theta_s $ is a martingale, then for all $ \xi'\in L^2_{\mathcal{F}^{\tilde{\mathbf{Y}}^\pi}}([0,T];V^\Theta(\cdot,\tilde{\mathbf{Y}}^\pi_\cdot)),\  \mathbb{E}^{\mathbb{P}^{W,N}}\int_{0}^{T}\xi'_t\d M_t^\Theta=0$, so we get (\ref{eq:martingale_orthogonality}).  
	
	On the other hand, by It\^{o} formula, we have
	\begin{equation}
		dV^\Theta\left(t,\tilde{\mathbf{Y}}_t^\pi\right)=\int_{\mathcal{A}}\pi_t(v)\mathcal{L}^{\pi(v)}V^\Theta\left(t,\tilde{\mathbf{Y}}_t^\pi\right)\d v\d t+\dfrac{\partial V^\Theta}{\partial S}\sigma \d W_t+\Delta V^\Theta \d \hat{N}_t,
		\label{eq:dvalue_theta}
	\end{equation} 
	where $ \Delta V^\Theta=\mathbb{E}[V^\Theta(t,\tilde{\mathbf{Y}}_t^\pi+\mathbf{y}_\eta)- V^\Theta(t,\tilde{\mathbf{Y}}_t^\pi)]$, $ \mathbf{y}_\eta=(0,0,0,\eta)^{\mathsf{T}} $, and $ \hat{N}=\{\hat{N}_t\}_{0\leq t\leq T} $ is the compensated Poisson process related to $ N $. According to Assumption \ref{assumption:parameterized_V}, the integral terms about $ dW_t $ and $ d\hat{N}_t $ are martingales, so their expectations are zero. Specially, let 
	$ \xi_t=\text{sgn}\left(r(\mathbf{Y}_t^\pi,v^\pi_t)-\gamma\log(\pi_t(v^\pi_t))+\int_{\mathcal{A}}\pi_t(v^\pi_t)\mathcal{L}^{\pi(v)}V^\Theta(t,\mathbf{Y}_t^\pi)\d v\right), $ we know $ \xi\in L^2_{\mathcal{F}^{{\mathbf{Y}}^\pi}}([0,T];V(\cdot,{\mathbf{Y}}^\pi_\cdot)) $ because of Assumption \ref{assumption1} and Definition \ref{admi_def}. Due to the condition that $ {\mathbf{Y}}^\pi $ and $ \tilde{\mathbf{Y}}^\pi $ have the same distribution again, we get 
	\begin{equation*}
		\begin{aligned}
			0&=\mathbb{E}^{\mathbb{P}}\int_{0}^T\xi_t\left[\d V^\Theta(t,\mathbf{Y}_t^\pi)+r(\mathbf{Y}_t^\pi,v^\pi_t)\d t-\gamma\log(\pi_t(v^\pi_t))\d t\right]\\
			&=\mathbb{E}^{\mathbb{P}^{W,N}}\int_{0}^T\xi'_t\left\{\d V^\Theta\left(t,\tilde{\mathbf{Y}}_t^\pi\right)+\int_{\mathcal{A}}\pi_t(v)\left[r\left(\tilde{\mathbf{Y}}_t^\pi,v\right)-\gamma\log(\pi_t(v))\right]\d v\d t\right\}\\
			&=\mathbb{E}^{\mathbb{P}^{W,N}}\int_{0}^T\left|\int_{\mathcal{A}}\bigg[r({\tilde{\mathbf{Y}}}_t^\pi,v)-\gamma\log(\pi_t(v))+\mathcal{L}^{\pi(v)}V^\Theta(t,{\mathbf{Y}}_t^\pi)\bigg]\pi_t(v)\d v\right|\d t,
		\end{aligned}
	\end{equation*}
	which follows that $ \int_{\mathcal{A}}[r({\tilde{\mathbf{Y}}}_t^\pi,v)-\gamma\log(\pi_t(v))+\mathcal{L}^{\pi(v)}V^\Theta(t,{\mathbf{Y}}_t^\pi)]\pi_t(v)\d v=0 $ almost everywhere.  
	Based on (\ref{eq:dvalue_theta}), we know $ \d M_t^\Theta=\dfrac{\partial V^\Theta}{\partial S}\sigma \d W_t+\Delta V^\Theta \d \hat{N}_t$, so $ \{M_t^\Theta\}_{0\leq t\leq T} $ is also a martingale due to $ W $ and $ \hat{N} $ being martingales. 
\end{proof}

The second method for policy evaluation is to solve the martingale orthogonality condition in Proposition \ref{prop:martingale_orthogonality}. Specifically, we need to choose a test function and solve (\ref{eq:martingale_orthogonality}), which is commonly associated with the Temporal Difference (TD) learning method in RL. The common method to solve this equation is random approximation. We denote $\alpha_\Theta^{(2)}$ as the learning rate for martingale orthogonality condition, so $\Theta$ can be updated in the following ways,  
\begin{equation*}
	\Theta\leftarrow\Theta+\alpha_\Theta^{(2)}\int_{0}^T\xi_t\d M_t^{\Theta}\approx\Theta+\alpha_\Theta^{(2)}\sum\limits_{i=0}^{N-1}\xi_{t_i}\left(M^{\Theta}_{t_{i+1}}-M^{\Theta}_{t_i}\right). 
\end{equation*}
In the first method, the martingale loss function will descend along the gradient, but we can't determine the sign of $ \alpha_\Theta^{(2)} $ since we don't know much about this equation near the true solution. Therefore, we make some modifications to the algorithm basing on \cite{jia2022policy}, which means we consider an equivalent minimization problem shown in (\ref{eq:min_MO}).
\begin{equation}
	\label{eq:min_MO}
	\min\limits_{\Theta}\left(\mathbb{E}^\mathbb{P}\int_{0}^{T}\xi_t\left[\d V^{\Theta}+r(\mathbf{Y}_t^{\pi^\Phi},v_t^{\pi^\Phi})\d t-\gamma\log(\pi^\Phi(v_t^{\pi^\Phi}\mid t,\mathbf{Y}_{t}^{\pi^\Phi}))\d t\right]
	\right)^2. 
\end{equation}
After discretization of time terminal, we get from one sample, 
\begin{equation}
	\label{MO_update}
	\min\limits_{\Theta}\bigg(\sum\limits_{i=0}^{N-1}\xi_{t_i}\bigg[\Delta V^{\Theta,\Phi}_{t_i}+r\left(\mathbf{Y}_{t_i}^{\pi^\Phi},v_{t_i}^{\pi^\Phi}\right)h
	-\gamma\log\left(\pi^\Phi(v_{t_i}^{\pi^\Phi}\mid t_i,\mathbf{Y}_{t_i}^{\pi^\Phi})\right)h\bigg]
	\bigg)^2,
\end{equation}
where $ \Delta V^{\Theta,\Phi}_{t_i}=V^{\Theta}(t_{i+1},\mathbf{Y}_{t_{i+1}}^{\pi^\Phi})-V^{\Theta}(t_{i},\mathbf{Y}_{t_i}^{\pi^\Phi}) $. Then we calculate the gradient to get the direction: 
\begin{equation}
	\label{eq:MO_update}
	\Delta \Theta=\sum\limits_{i=0}^{N-1}\xi_{t_i}\bigg[\Delta V^{\Theta,\Phi}_{t_i}+r\left(\mathbf{Y}_{t_i}^{\pi^\Phi},v_{t_i}^{\pi^\Phi}\right)h-\gamma\log\left(\pi^\Phi(v_{t_i}^{\pi^\Phi}| t_i,\mathbf{Y}_{t_i}^{\pi^\Phi})\right)h\bigg]\dfrac{\partial\Delta V^{\Theta,\Phi}_{t_i}}{\partial \Theta}, 
\end{equation}
and update by the following formula with the learning rate $ \alpha_\Theta^{(2)} $, which is a positive constant, 
\begin{equation}
	\Theta\leftarrow\Theta-\alpha_\Theta^{(2)}\Delta\Theta.
\end{equation}
Additionally, we also provide a feasible method to choose test functions, and take the terminal condition into consideration. According to Proposition \ref{prop:martingale_orthogonality}, we need to choose $ \xi\in L^2_{\mathcal{F}^{\mathbf{Y}^\pi}}([0,T];V^\Theta(\cdot,\mathbf{Y}^\pi_\cdot,\pi)) $. We follow uniform distribution to choose a random number at $ t_i $ as the value of $\xi_{t_i}$, set $ \xi_{t_{N-1}}=1 $ to amplify the impact of real terminal conditions in the loss function, and replace $ V^{\Theta}(T) $ at terminal with the real sampling data $ V^{k}(T) $ to take the terminal condition into consideration. The superscript $ k $ represents the real return based on the sample trajectory in the $k^{th}$ episode. 

\subsection{Policy improvement}
To update the policy, noticing that the only parameter in the policy is $ \Phi $, our target is to find the policy gradient $ g(t,\mathbf{y},\Phi):=\frac{\partial}{\partial\Phi}V^{\pi^\Phi}(t,\mathbf{y}) $. Substitute the parameterized policy $\pi^\Phi$ into (\ref{feymann}) and take its derivative with respect to $ \Phi $, then we have
\begin{multline}
	\int_{\mathcal{A}}\left\{\pi^\Phi(v|t,\mathbf{y})\bigg[\mathcal{L}^{\pi^\Phi(v)}g(t,\mathbf{y},\Phi)-\gamma\dfrac{\partial}{\partial\Phi}\log(\pi^\Phi(v\mid t,\mathbf{y}))\bigg]\right.\\ +\dfrac{\partial}{\partial\Phi}\pi^\Phi(v|t,\mathbf{y})\bigg[\mathcal{L}^{\pi^\Phi(v)}V^{\pi^\Phi}(t,\mathbf{y},\pi^\Phi)+r(\mathbf{y},v)-\gamma\log(\pi^\Phi(v\mid t,\mathbf{y}))\bigg]
	\bigg\}\d v=0,
	\label{partial}
\end{multline}
which also satisfies the terminal condition $ g(T,\mathbf{y},\Phi)=0 $. Define
\begin{equation}
	\label{eq:check_r}
	\begin{aligned}
		\check{r}(\mathbf{y},v,\Phi)=\dfrac{\frac{\partial\pi^\Phi}{\partial\Phi}(v|t,\mathbf{y})}{\pi^\Phi(v|t,\mathbf{y})}\bigg[&\mathcal{L}^{\pi^\Phi(v)}V^{\pi^\Phi}(t,\mathbf{y})+r(\mathbf{y},v)\\
		&-\gamma\log(\pi^\Phi(v\mid t,\mathbf{y}))\bigg]-\gamma\dfrac{\partial}{\partial\Phi}\log(\pi^\Phi(v\mid t,\mathbf{y})),
	\end{aligned}
\end{equation}
and then we rewrite (\ref{partial}) as 
\begin{equation}
	\int_{\mathcal{A}}\pi^\Phi(v\mid t,\mathbf{y})\bigg[\mathcal{L}^{\pi^\Phi(v)}g(t,\mathbf{y},\Phi)+\check{r}(\mathbf{y},v,\Phi)\bigg]\d v=0,\qquad g(T,\mathbf{y},\Phi)=0.
	\label{new_fey}
\end{equation}
Consequently, we apply Feynman-Kac formula and get another expression of $ g $, 
\begin{equation}
	\label{grad_feynman}
	\begin{aligned}
		g(t,\mathbf{y},\Phi)&=\mathbb{E}^{\mathbb{P}}\left.\left[\int_{t}^{T}\check{r}(\mathbf{Y}^{\pi^\Phi}_s,v^{\pi^\Phi}_s,\Phi)\d s\right|\mathbf{Y}^{\pi^\Phi}_t=\mathbf{y}\right]\\
		&=\mathbb{E}^{\mathbb{P}^{W,N}}\left.\left[\int_{t}^T\int_{\mathcal{A}}\pi^\Phi(v\mid s,\tilde{\mathbf{Y}}_s^{\pi^\Phi})\check{r}(\tilde{\mathbf{Y}}^{\pi^\Phi}_s,v,\Phi)\d v\d s\right|\tilde{\mathbf{Y}}^{\pi^\Phi}_t=\mathbf{y}\right].  
	\end{aligned}
\end{equation}
Then the calculation of policy gradient can be done in the same way as shown in policy evaluation, with the difference of reward functions. However, the new reward function contains the item $\mathcal{L}^{\pi^\Phi(v)}V^{\pi^{\Phi}}$, which is not a direct result through sampling, so further processing is needed. First, we need the following technical assumptions for the policy estimators. 
\begin{assumption}
	\label{check_r_assumption}
	The following conditions hold true:
	\begin{enumerate}
		\item For all $(t,\mathbf{y})\in[0,T]\times\mathbb{R}^4,\  \pi^\Phi(v\mid t,\mathbf{y}) $ is smooth about $ \Phi $, and there exists $ C>0,\ p\geq1 $, such that $ \int_{\mathcal{A}}|\check{r}(\mathbf{y},v,\Phi)|\pi^\Phi(v|t,\mathbf{y})\d v\leq C(1+|\mathbf{y}|^p) $; 
		\item $ \int_{\mathcal{A}}|\frac{\partial}{\partial \Phi}\log(\pi^\Phi(v\mid t,\mathbf{y}))|^2\pi^\Phi(v\mid t,\mathbf{y})\d v $ is continuous in $(t,\mathbf{y})$;
		\item $ V^{\pi^\Phi}(t,\mathbf{y}) \in C^{1,2}([0,T]\times\mathbb{R}^4)$ is polynomial growth in $ \mathbf{y} $.    
	\end{enumerate}
\end{assumption}
The following theorem provides the formula of the policy gradient $ g(t,\mathbf{y},\Phi) $. 
\begin{theorem}
	Given an admissible policy $\pi^\Phi$, $ g(t,\mathbf{y},\Phi)=\frac{\partial}{\partial \Phi}V^{\pi^\Phi}(t,\mathbf{y}) $ is the gradient of it and satisfies \
	\begin{equation}
		\label{policy_g}
		\resizebox{.90\hsize}{!}{$\begin{aligned}
				g(t,\mathbf{y},\Phi)=\mathbb{E}^{\mathbb{P}}\bigg[
				\int_{t}^T\dfrac{\partial}{\partial \Phi}\log(\pi^\Phi(v^{\pi^\Phi}_s|s,\mathbf{Y}^{\pi^\Phi}_s))\bigg(\d V^{\pi^\Phi}(s,\mathbf{Y}^{\pi^\Phi}_s)+\tilde{r}^\Phi_s\d s\bigg)\bigg|\mathbf{Y}^{\pi^\Phi}_t=\mathbf{y}\bigg],
			\end{aligned} $}
	\end{equation}
	where $\tilde{r}^\Phi_s =r(\mathbf{Y}^{\pi^\Phi}_s,v^{\pi^\Phi}_s)-\gamma\log(\pi^\Phi(v^{\pi^\Phi}_s\mid s,\mathbf{Y}^{\pi^\Phi}_s))-\gamma $. 
\end{theorem}
\begin{proof}
	We only need to prove (\ref{grad_feynman}) and (\ref{policy_g}) are equivalent. For fixed $ t $, Define $ \tau_n=\inf\{s\geq t:|\mathbf{Y}^{\pi^\Phi}_s|\geq n\} $. Due to the definition of $\check{r}$ in (\ref{eq:check_r}), we apply It\^{o} formula to $ V^{\pi^\Phi}(s,\mathbf{Y}^{\pi^\Phi}_s) $, and then get
	\begin{equation*}\resizebox{.99\hsize}{!}{$ \begin{aligned}
				&\int_{t}^{T\wedge\tau_n}\dfrac{\partial}{\partial \Phi}\log(\pi^\Phi(v^{\pi^\Phi}_s|s,\mathbf{Y}^{\pi^\Phi}_s))\bigg\{\d V^{\pi^\Phi}(s,\mathbf{Y}^{\pi^\Phi}_s)+\tilde{r}^\Phi_s\d s\bigg\}\\
				=&\int_{t}^{T\wedge\tau_n}\dfrac{\partial}{\partial \Phi}\log(\pi^\Phi(v^{\pi^\Phi}_s|s,\mathbf{Y}^{\pi^\Phi}_s))\bigg\{\mathcal{L}^{\pi^\Phi(v)}V^{\pi^\Phi}(s,\mathbf{Y}^{\pi^\Phi}_s)+\tilde{r}^\Phi_s\d s+\dfrac{\partial V^{\pi^\Phi}}{\partial S}\sigma \d W_s+\Delta V^{\pi^\Phi} \d \hat{N}_s\bigg\},\\
				=& \int_{t}^{T\wedge\tau_n}\bigg\{\check{r}(\mathbf{Y}^{\pi^\Phi}_s,v^{\pi^\Phi}_s,\Phi)ds +\dfrac{\partial}{\partial \Phi}\log(\pi^\Phi(v^{\pi^\Phi}_s|s,\mathbf{Y}^{\pi^\Phi}_s))\bigg[\dfrac{\partial V^{\pi^\Phi}}{\partial S}\sigma \d W_s+\Delta V^{\pi^\Phi}\d \hat{N}_s\bigg]\bigg\}, 
			\end{aligned} $}
	\end{equation*}
	where $ \Delta V^{\pi^\Phi}=\mathbb{E}[V^{\pi^\Phi}(s,\mathbf{Y}^{\pi^\Phi}_s+\mathbf{y}_\eta)-V^{\pi^\Phi}(s,\mathbf{Y}^{\pi^\Phi}_s)] $.
	Notice that
	\begin{equation*}
		\begin{aligned}
			&\mathbb{E}^\mathbb{P}\left[\int_{t}^{T\wedge\tau_n}\dfrac{\partial}{\partial \Phi}\log(\pi^\Phi(v^{\pi^\Phi}_s|s,\mathbf{Y}^{\pi^\Phi}_s))\left\{\dfrac{\partial V^{\pi^\Phi}}{\partial S}\sigma \d W_s+\Delta V^{\pi^\Phi} \d \hat{N}_s
			\right\}\bigg|\mathbf{Y}^{\pi^\Phi}_t=\mathbf{y}\right]\\
			=&\mathbb{E}^{\mathbb{P}^{W,N}}\bigg[\int_{t}^{T\wedge\tilde{\tau}_n}\int_{\mathcal{A}}\dfrac{\partial}{\partial \Phi}\log(\pi^\Phi(v\mid s,\tilde{\mathbf{Y}}^{\pi^\Phi}_s))\pi^\Phi(v|s,\tilde{\mathbf{Y}}^{\pi^\Phi}_s)\\
			&\qquad\qquad\qquad\qquad\times\bigg(\dfrac{\partial V^{\pi^\Phi}}{\partial S}\sigma \d W_s+\Delta V^{\pi^\Phi} \d \hat{N}_s
			\bigg)\d v\bigg|\tilde{\mathbf{Y}}^{\pi^\Phi}_t=\mathbf{y}
			\bigg], 
		\end{aligned}
	\end{equation*}
	where $ \tilde{\tau}_n=\inf\{s\geq t:|\tilde{\mathbf{Y}}^{\pi^\Phi}|\geq n\} $ because $ \tilde{\mathbf{Y}}^{\pi^\Phi} $ and $ {\mathbf{Y}}^{\pi^\Phi} $ have the same distribution. Before going further, we check the square integrability of processes, which are $$ \int_{t}^{T\wedge\tilde{\tau}_n}\left|\int_{\mathcal{A}}\dfrac{\partial}{\partial \Phi}\log(\pi^\Phi(v\mid s,\tilde{\mathbf{Y}}^{\pi^\Phi}_s))\dfrac{\partial V^{\pi^\Phi}}{\partial S}\sigma\pi^\Phi(v\mid s,\tilde{\mathbf{Y}}^{\pi^\Phi}_s) \d v\right|^2\d s<\infty, $$
	and $$ \int_{t}^{T\wedge\tilde{\tau}_n}\left|\int_{\mathcal{A}}\dfrac{\partial}{\partial \Phi}\log(\pi^\Phi(v|s,\tilde{\mathbf{Y}}^{\pi^\Phi}_s))\Delta V^{\pi^\Phi}\pi^\Phi(v|s,\tilde{\mathbf{Y}}^{\pi^\Phi}_s) \d v\right|^2\d s<\infty. $$
	Since the integral's interval is bounded, we only need to prove the boundedness of the integrand. According to the definition of $\tau_n$ and Assumption \ref{check_r_assumption}, we have
	\begin{equation*}	 
		\begin{aligned}
			&\left|\int_{\mathcal{A}}\dfrac{\partial}{\partial \Phi}\log(\pi^\Phi(v\mid s,\tilde{\mathbf{Y}}^{\pi^\Phi}_s))\dfrac{\partial V^{\pi^\Phi}}{\partial S}\sigma\pi^\Phi(v\mid s,\tilde{\mathbf{Y}}^{\pi^\Phi}_s) \d v\right|^2\\
			\leq&\int_{\mathcal{A}}\left|\dfrac{\partial}{\partial \Phi}\log(\pi^\Phi(v\mid s,\tilde{\mathbf{Y}}^{\pi^\Phi}_s))\dfrac{\partial V^{\pi^\Phi}}{\partial S}\sigma\right|^2\pi^\Phi(v\mid s,\tilde{\mathbf{Y}}^{\pi^\Phi}_s) \d v\\
			\leq& C\max\limits_{0\leq t\leq T\wedge\tilde{\tau}_n}\left|\dfrac{\partial V^{\pi^\Phi}}{\partial S}\right|\int_{\mathcal{A}}\left|\dfrac{\partial}{\partial \Phi}\log(\pi^\Phi(v\mid s,\tilde{\mathbf{Y}}^{\pi^\Phi}_s))\right|^2\pi^\Phi(v\mid s,\mathbf{Y}^{\pi^\Phi}_s)\d v< \infty. 
		\end{aligned}
	\end{equation*}
	The similar property of the integrand about $ d\hat{N}_s $ is proved directly according to Assumption \ref{check_r_assumption}. Hence these two integral terms are equal to 0 after taking expectation on previous formula, which gives 
	\begin{equation*}
		\begin{aligned}
			&\mathbb{E}^\mathbb{P}\bigg[\int_{t}^{T\wedge\tau_n}\dfrac{\partial}{\partial \Phi}\log(\pi^\Phi(v^{\pi^\Phi}_s|s,\mathbf{Y}^{\pi^\Phi}_s))\bigg(\d V^{\pi^\Phi}(s,\mathbf{Y}^{\pi^\Phi}_s)+\tilde{r}^\Phi_{s}\d s\bigg)
			\bigg|\mathbf{Y}^{\pi^\Phi}_t=\mathbf{y}\bigg]\\
			=&\mathbb{E}^\mathbb{P}\bigg[\int_{t}^{T\wedge\tau_n}\check{r}(\mathbf{Y}^{\pi^\Phi}_s,v^{\pi^\Phi}_s,\Phi)\d s\bigg|\mathbf{Y}^{\pi^\Phi}_t=\mathbf{y}\bigg].
		\end{aligned}
	\end{equation*}
	Finally using Theorem \ref{thm:growth_control}, we know that $ \lim\limits_{n\rightarrow\infty}T\wedge\tau_n=T $. According to Assumption \ref{check_r_assumption} and dominated convergence theorem, let $ n\rightarrow\infty $ and we obtain the conclusion. 
\end{proof}

Since (\ref{policy_g}) can be approximated through sampling, we use $ V^\Theta $ to approximate $ V^{\pi^\Phi} $, and get the approximate calculation of the policy gradient to update our policy with learning rate $ \alpha_\Phi>0 $. More precisely, we update $ \Phi $ by the following formula, 
\begin{equation}
	\Phi\leftarrow\Phi+\alpha_\Phi \sum\limits_{i=0}^{N-1}\Delta \Phi_i,  
\end{equation}
\begin{equation}
	\label{update_policy}
	\Delta \Phi_i=\dfrac{\partial}{\partial\Phi}\log(\pi^\Phi(v^{\pi^\Phi}_{t_i}|t_i,\mathbf{Y}_{t_i}^{\pi^\Phi}))[V^{\Theta}({t_{i+1}},\mathbf{Y}^{\pi^\Phi}_{t_{i+1}})-V^{\Theta}({t_{i}},\mathbf{Y}^{\pi^\Phi}_{t_{i}})+\tilde{r}^\Phi_{t_i}h
	]. 
\end{equation}

So far we have completed two main steps of RL: policy evaluation and policy improvement. There are two different methods for policy evaluation, corresponding to martingale loss function and martingale orthogonality condition. In our Actor-Critic algorithms, we use (\ref{update_policy}) for policy improvement. Combining policy evaluation and policy improvement we propose two Actor-Critic algorithms, and for simplicity, we call the algorithm that using martingale loss function for policy evaluation ML-AC, and the algorithm that using martingale orthogonality condition MO-AC. We present pseudo codes in Algorithm \ref{algorithm:ML-AC} and Algorithm \ref{algorithm:MO-AC}. 

\begin{algorithm}[ht]
	\caption{Actor-Critic Algorithm with Martingale Loss Function (ML-AC)}
	\label{algorithm:ML-AC}
	\begin{algorithmic}[1]
		\renewcommand{\algorithmicrequire}{\textbf{Input}}
		\Require{Environment $ Env $, initial state $ \mathbf{y}_0=(X_0,S_0,\mu_0,\mathfrak{N}) $, terminal time $ T $, time step $ h $, number of time
			grids $ G $, target penalty parameter $\varphi$, temperature parameter $\gamma$, initial learning rate $ \alpha_\Theta$ and $ \alpha_\Phi^{(1)} $, number of episode $ M $ and learning rate schedule function $ l(\cdot) $;}
		\State{Initialize action network $ \pi^\Phi $ and critic network $ V^\Theta $; }
		\For{$ m=1:M $}
		\For{$ i=0:G-1 $}
		\State{At time $ t_i=ih $, compute the value of current state $ V^\Theta(t_i,\mathbf{Y}_{t_i}) $}
		\State{Generate an action $ v_{t_i}^{\pi^\Phi} $ following the Gaussian distribution $ \pi^\Phi(\cdot\mid t_i,\mathbf{Y}_{t_i}) $; }
		\State{Take the action $ v_{t_i}^{\pi^\Phi} $, observe new state $ \mathbf{Y}_{t_{i+1}} $ and receive the reward $ r_{t_i} $ from $ Env $; }
		\EndFor
		\State{At time $ T $, clear all inventory and obtain the terminal reward $ f(\mathbf{Y}_T) $; }
		\State{Compute $ \Delta\Theta $ by (\ref{eq:ML_update}) and $ \Delta\Phi $ by (\ref{update_policy}); }
		\State{Update the critic network by $ \Theta\leftarrow\Theta+l(m)\alpha_\Theta^{(1)}\Delta\Theta $; }
		\State{Update the action network by $ \Phi\leftarrow\Theta+l(m)\alpha_\Phi\Delta\Phi $; }
		\EndFor
	\end{algorithmic}
\end{algorithm}

\begin{algorithm}[ht]
	\caption{Actor-Critic Algorithm with Martingale Orthogonality Condition (MO-AC)}
	\label{algorithm:MO-AC}
	\begin{algorithmic}[1]
		\renewcommand{\algorithmicrequire}{\textbf{Input}}
		\Require{Environment $ Env $, initial state $ \mathbf{y}_0=(X_0,S_0,\mu_0,\mathfrak{N}) $, terminal time $ T $, time step $ h $, number of time
			grids $ G $, target penalty parameter $\varphi$, temperature parameter $\gamma$, initial learning rate $ \alpha_\Theta^{(2)}$ and $ \alpha_\Phi $, number of episode $ M $ and learning rate schedule function $ l(\cdot) $;}
		\State{Initialize action network $ \pi^\Phi $ and critic network $ V^\Theta $; }
		\For{$ m=1:M $}
		\For{$ i=0:G-1 $}
		\State{At time $ t_i=ih $, compute the value of current state $ V^\Theta(t_i,\mathbf{Y}_{t_i}) $; }
		\State{Generate an action $ v_{t_i}^{\pi^\Phi} $ from the Gaussian distribution $ \pi^\Phi(\cdot|t_i,\mathbf{Y}_{t_i}) $; }
		\State{Generate $ \xi_{t_i} $ following the uniform distribution on $ (0,1) $; }
		\State{Take the action $ v_{t_i}^{\pi^\Phi} $, observe new state $ \mathbf{Y}_{t_{i+1}} $ and receive the reward $ r_{t_i} $ from $ Env $; }
		\EndFor
		\State{At time $ T $, clear all inventory and obtain the terminal reward $ f(\mathbf{Y}_T) $; }
		\State{Compute $ \Delta\Theta $ by (\ref{eq:MO_update}) and $ \Delta\Phi $ by (\ref{update_policy}); }
		\State{Update the critic network by $ \Theta\leftarrow\Theta-l(m)\alpha_\Theta^{(2)}\Delta\Theta $; }
		\State{Update the action network by $ \Phi\leftarrow\Theta+l(m)\alpha_\Phi\Delta\Phi $; }
		\EndFor
	\end{algorithmic}
\end{algorithm}

\section{Numerical simulation}
\label{section_result}
In this section, we present how we create market simulators and test our algorithms in two different environments. 
\subsection{Environment simulation}
\label{env_simulation}
\subsubsection{Simulations of market trading speed}
\label{subsubsection:mu_simulation}
Referring to the settings in \cite{cartea2015algorithmic}, we set $ \eta $ to follow an exponential distribution with parameter $ \eta_0 $. 
Glassermann has introduced two methods for simulating the geometric Brownian motion with jumps \cite{glasserman2004monte}, which can be applied here, and we choose one of it to create paths of processes. 

Fix time discretizing points $ 0=t_0<t_1<...<t_n=T $, consider the discrete form of (\ref{trading_speed}): $ \mu_{t_{i+1}}-\mu_{t_i}=-\kappa\mu_{t_i}(t_{i+1}-t_i)+\sum\limits_{j=N_{t_i}+1}^{N_{t_{i+1}}}\eta_j $, and then we follow steps below to generate samples from $ t_i $ to $ t_{i+1} $: 
\begin{enumerate}
	\item Generate $ \tilde{N} $ following a Poisson distribution with intensity $ \lambda(t_{i+1}-t_i) $. If $ \tilde{N}=0 $, then let $ \tilde{M}=0 $ and jump to step 3; 
	\item Generate $ \eta_1,...,\eta_{\tilde{N}} $ from their distribution and let $ \tilde{M}=\eta_1+...+\eta_{\tilde{N}} $;
	\item Let $ \mu_{t_{i+1}}=(1-\kappa(t_{i+1}-t_i))\mu_{t_i}+\tilde{N} $.
\end{enumerate}
This method mainly relies on the properties of Poisson process. The left in Figure \ref{fig:10samples} shows ten samples of the market trading speed process gained in this way. 
\subsubsection{Simulation of the interaction between the agent and environment}
We need to simulate the market based on (\ref{enviroment}), so first set six environment parameters $ (\kappa_{env},\ \lambda_{env},\ \eta_{env},\ k_{env},\ b_{env},\ \alpha_{env}) $, which are unknown for the agent. Besides, there are three known penalty parameters need to be set:  $ (\varphi, \rho, \gamma) $.

The agent has knowledge of $ (S_t,X_t,Q_t,\mu_t) $ at time $ t $, determines the trading speed $ v $ based on it, and then gets the target penalty: $ -\varphi(v-\rho\mu_t)^2h $. In the same time, the environment receives the action and update the state:  
\begin{equation*}
	\begin{aligned}
		&S_{t+h}=S_t-b_{env}vh+\sigma\sqrt{h}\tilde{Z}_t,\qquad \tilde{Z}_t\sim\mathcal{N}(0,1),\\
		&\hat{S}_{t+h}=S_{t+h}-k_{env}v,\qquad X_{t+h}=X_t+\hat{S}_{t+h}vh,\qquad Q_{t+h}=Q_t-vh, 
	\end{aligned}
\end{equation*}
and $ \mu_{t+h} $ is obtained as described in Subsection \ref{subsubsection:mu_simulation}. The right in Figure \ref{fig:10samples} shows ten samples of the stock price process while the agent takes the optimal policy. 

\begin{figure}[ht]
	\centering
	\begin{minipage}[t]{0.49\linewidth}
		\includegraphics[width=0.95\linewidth]{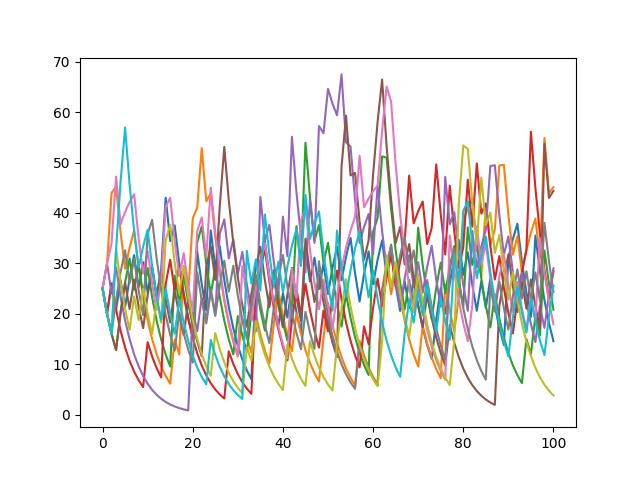}
	\end{minipage}
	\begin{minipage}[t]{0.49\linewidth}
		\includegraphics[width=0.95\linewidth]{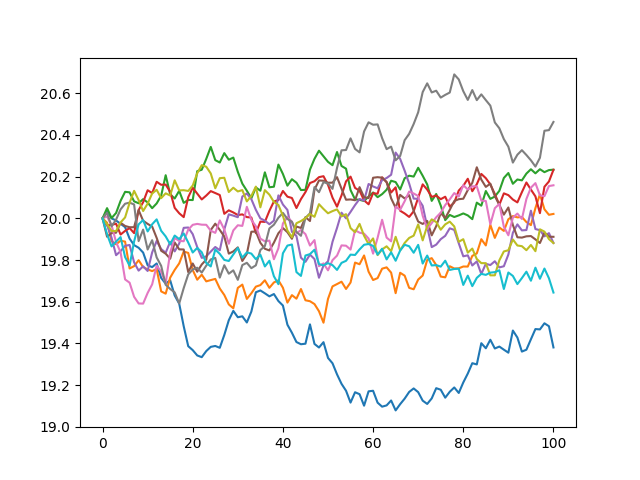}
	\end{minipage}
	\caption{Ten samples of market trading speed process (left) and stock price process (right). }
	\label{fig:10samples}
\end{figure}

\subsection{Parameter estimation}
\label{param_estimation}
Recall that we need to estimate some parameters by collecting data from the market estimator, since $\hat{b}$ and $\hat{k}$ are part of the loss function in ADP. 

In this model, we assume that the stock price process is driven by a Brownian motion $ W_t $ and satisfies $ \d S^v_t=-bv_t\d t+\sigma \d W_t $. We estimate the parameter $ b $ by the maximum likelihood estimation method. $ t_0, t_1,..., t_N $ are given observation points in time on $ [0,T] $, and we apply the Euler-Maruyama discretization to get $ S_{t_{i+1}}-S_{t_i}=-bv_{t_i}(t_{i+1}-t_i)+\sigma(W_{t_{i+1}}-W_{t_i}) $. Since $ S_{t_{i+1}}-S_{t_i}\sim\mathcal{N}(-bv_{t_i}(t_{i+1}-t_i),\sigma^2(t_{i+1}-t_i)) $, we obtain maximum likelihood functions of $ b $, which is
\begin{equation*}
	L(b)=\prod\limits_{i=0}^{N-1}\dfrac{1}{\sqrt{2\pi\sigma^2(t_{i+1}-t_i)}}\exp\left(-\dfrac{(S_{t_{i+1}}-S_{t_i}+bv_{t_i}(t_{i+1}-t_i))^2}{2\sigma^2(t_{i+1}-t_i)}\right). 
\end{equation*}
Take the logarithm of the above equation to have the logarithm likelihood function $ l(b) $,
calculate the partial derivative with respect to $ b $, and 
let $ \frac{\partial l}{\partial b}=0 $ to get the estimation of $ b $,  
\begin{equation*}
	\hat{b}=\dfrac{\sum\limits_{i=0}^{N-1}(S_{t_{i+1}}-S_{t_i})v_{t_i}}{\sum\limits_{i=0}^{N-1}v_{t_i}^2(t_{i+1}-t_i)}. 
\end{equation*}
Next we get the value of $ k $, which equals to $ (S_{t_i}-\hat{S}_{t_i})/v_{t_i} $. 

\subsection{Simulation results}
By using the simulation methods in Subsection \ref{env_simulation}, we are able to generate interactive market environment simulators. Referring to the parameter settings in Section 9.2 of \cite {cartea2015algorithmic}, we establish two different environments to test our algorithms, and the specific parameter settings are shown in Table \ref{param_table1} and Table \ref{param_table2}. 

In Environment 1, the impact of the agent's trading and target penalty parameter $\varphi$ are modest, while the the terminal penalty parameter $\alpha$ is relatively large. 
\begin{table}[H]
	\centering
	\caption{The parameter settings of Environment 1. }
	\begin{tabular}{|c|c|c|c|c|c|c|}
		\hline
		$ S_0 $ & $ Q_0 $ & $ X_0 $ & $ \mu_0 $ & $ b $ & $ k $ &$\varphi$\\ \hline
		20 & 1.25 & 0 & 25 & 0.1 & 0.1 &0.1\\ \hline
		$ \sigma $ & $ \alpha $ & $ \rho $ & $ \lambda $ & $ \eta_0 $ & $ \kappa $&$\gamma$ \\ \hline
		0.5 & 100 & 0.02 & 50 & 10 & 20 &0.001\\ \hline
	\end{tabular}
 \label{param_table1}
\end{table}
As a comparison, we set up an environment with larger price impacts and target penalty during the process of transaction, but lower terminal penalty parameter in Environment 2.  
\begin{table}[!ht]
	\centering
	\caption{The parameter settings of Environment 2. }
	\label{param_table2}
	\begin{tabular}{|c|c|c|c|c|c|c|}
		\hline
		$ S_0 $ & $ Q_0 $ & $ X_0 $ & $ \mu_0 $ & $ b $ & $ k $ &$\varphi$\\ \hline
		20 & 1.25 & 0 & 25 & 0.5 & 0.5 &10\\ \hline
		$ \sigma $ & $ \alpha $ & $ \rho $ & $ \lambda $ & $ \eta_0 $ & $ \kappa $ &$\gamma$\\ \hline
		0.5 & 10 & 0.02 & 50 & 10 & 20&0.001 \\ \hline
	\end{tabular}
\end{table}

We apply the same network structure in both environments. Specifically, we set up three hidden layers with 128, 64, and 32 units per layer in order, and a step function as reduction in the learning rate, which means reducing the learning rate by 0.9 after 
every 10 epochs. The input for both action and critic networks is the current state. The critic network output is the value of current state, and the action network output is a two-dimensional array, representing the mean value and the standard deviation of our policy. 
For each algorithm, we train it 1000 epoches with five different random seeds to test the universality of it. After each training session, we run five independent trials, with the solid curves corresponding to the mean and the shaded area corresponding to the minimum and maximum returns. We also compare our algorithm with SAC, a popular RL algorithm proposed by Haarnoja et al. \cite{haarnoja2018soft}, that introduces entropy regularization to encourage exploration.

\begin{figure}[H]
	\centering
	\subfigure[ADP] {\includegraphics[width=.32\textwidth]{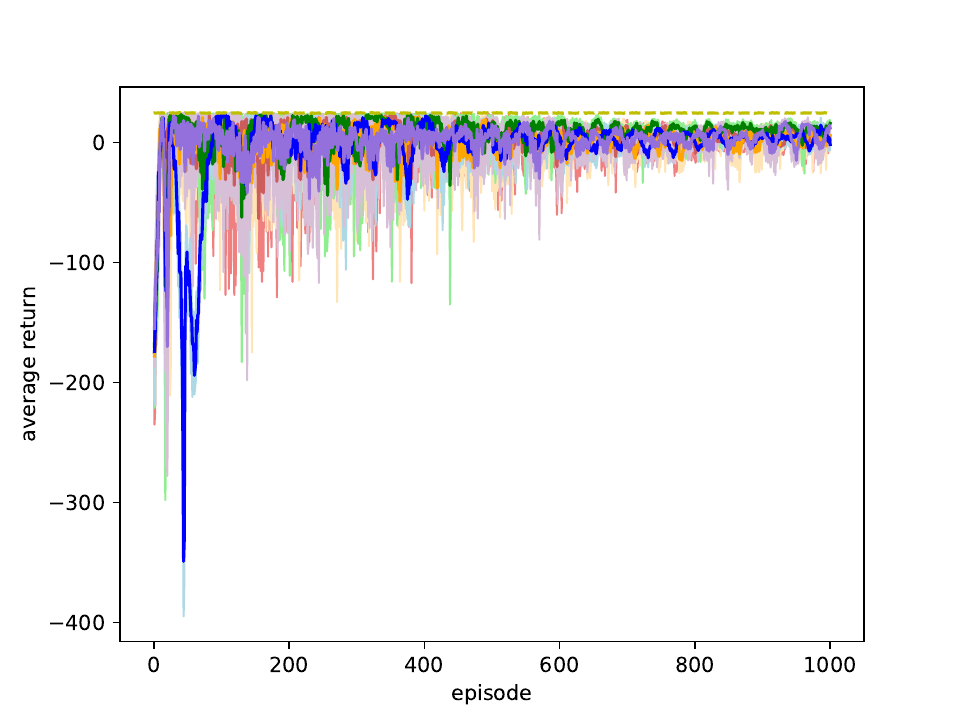}\label{pic:ADHDP_return1}}
	\subfigure[exploratory ADP] {\includegraphics[width=.32\textwidth]{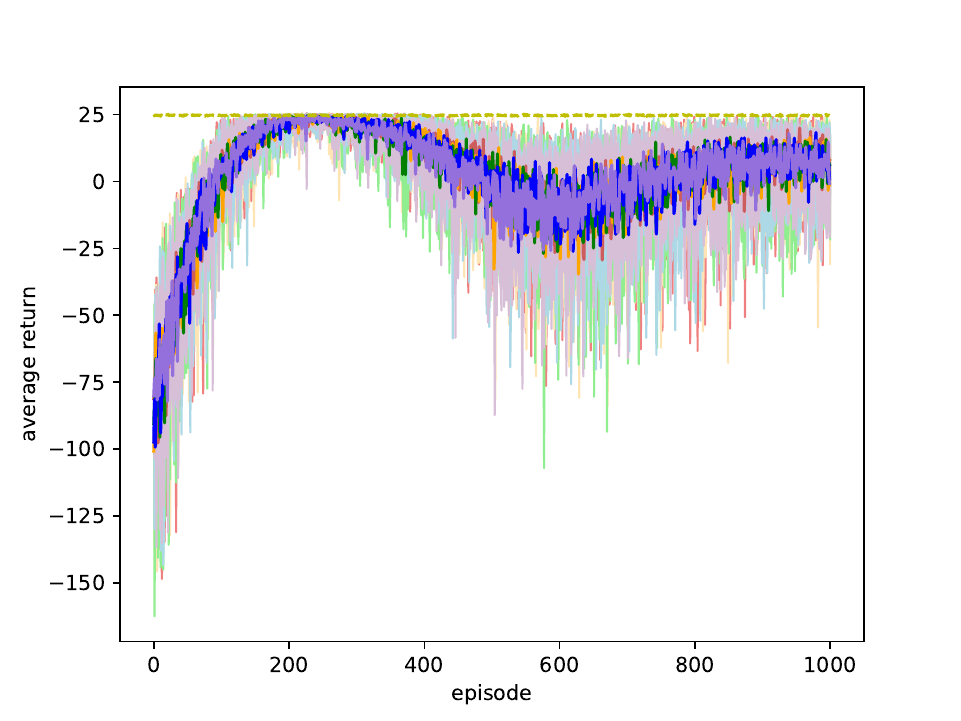}}
	\subfigure[SAC] {\includegraphics[width=.32\textwidth]{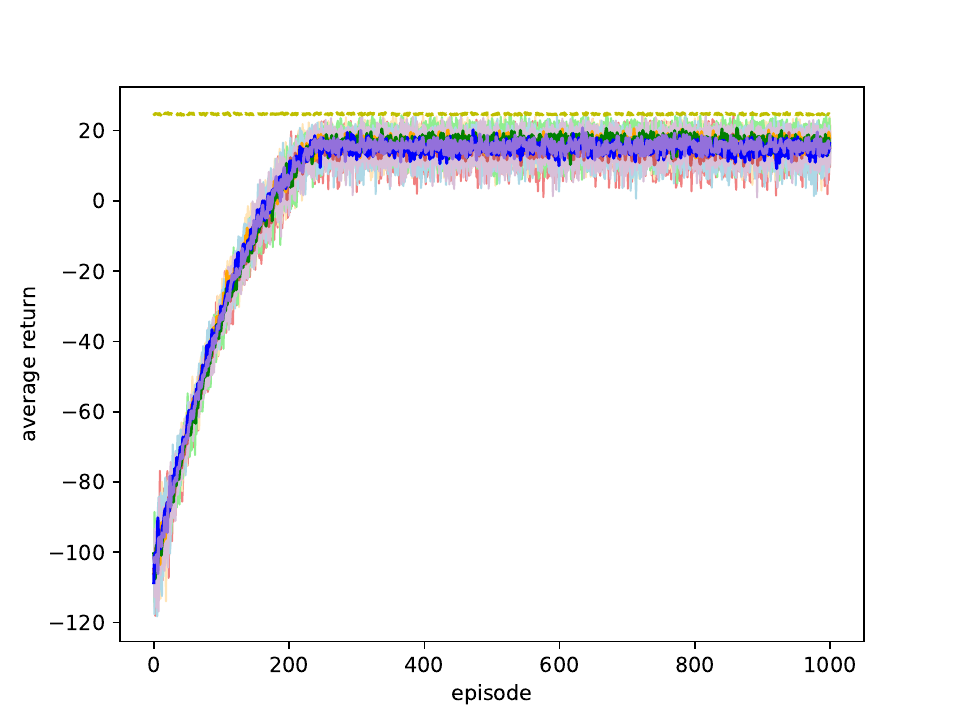}}

	\subfigure[ML-AC] {\includegraphics[width=.32\textwidth]{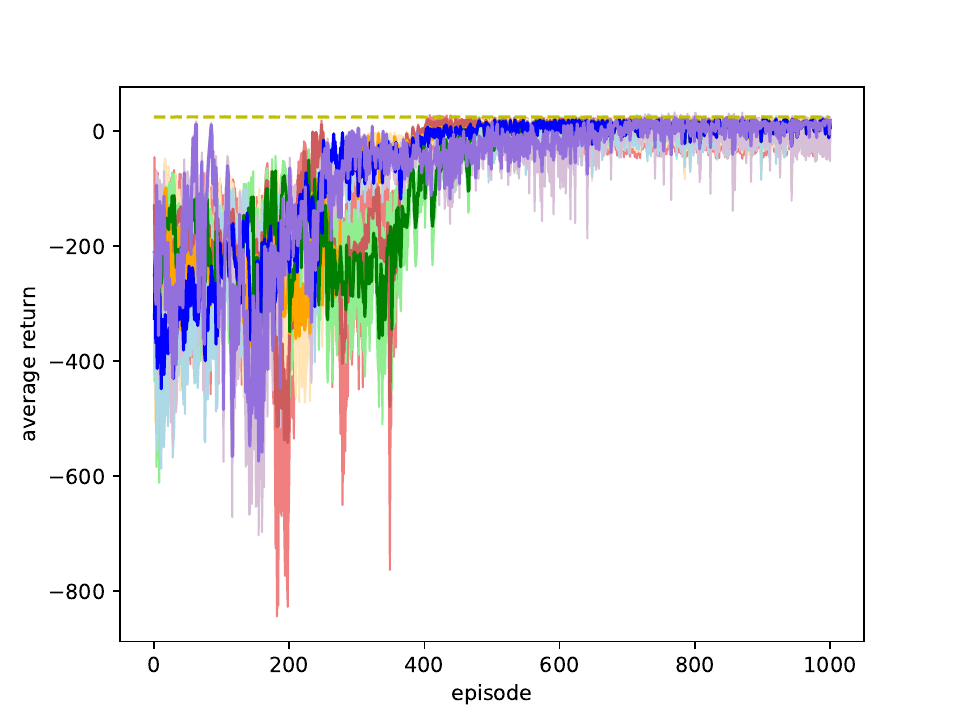}\label{pic:ML-AC_return1}}
	\subfigure[MO-AC] {\includegraphics[width=.32\textwidth]{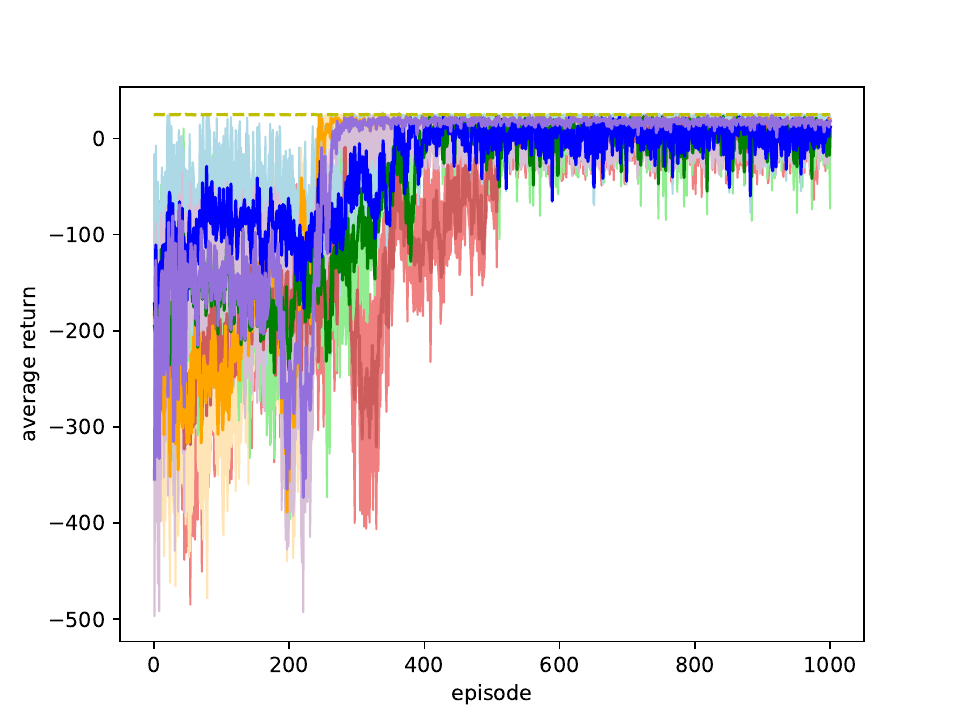}\label{pic:MO-AC_return1}}
	
	\subfigure[ADP] {\includegraphics[width=.32\textwidth]{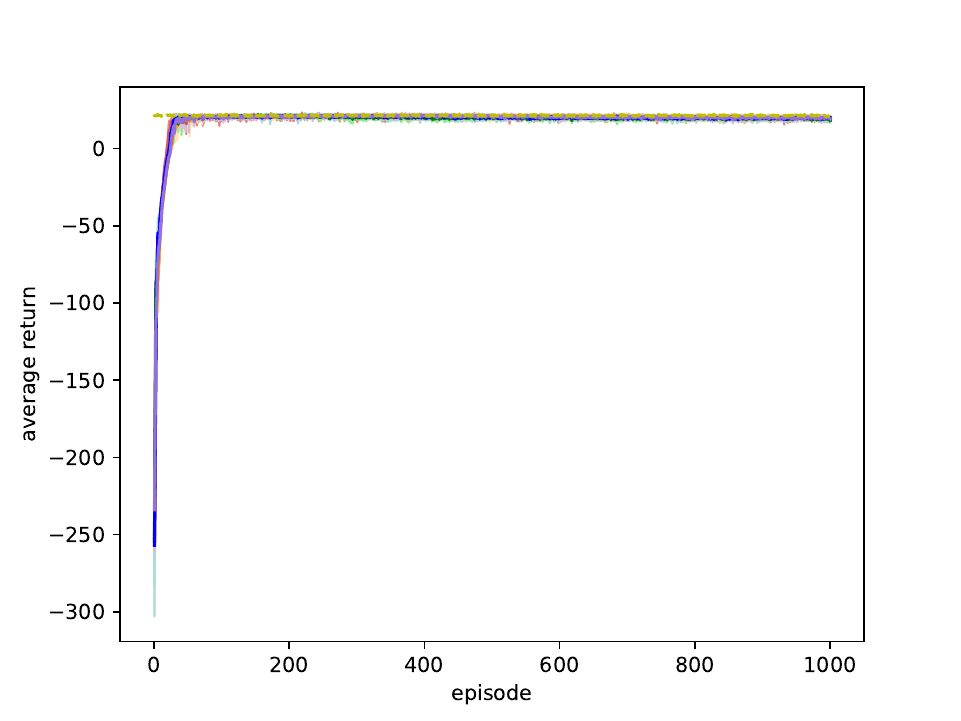}}
	\subfigure[exploratory ADP] {\includegraphics[width=.32\textwidth]{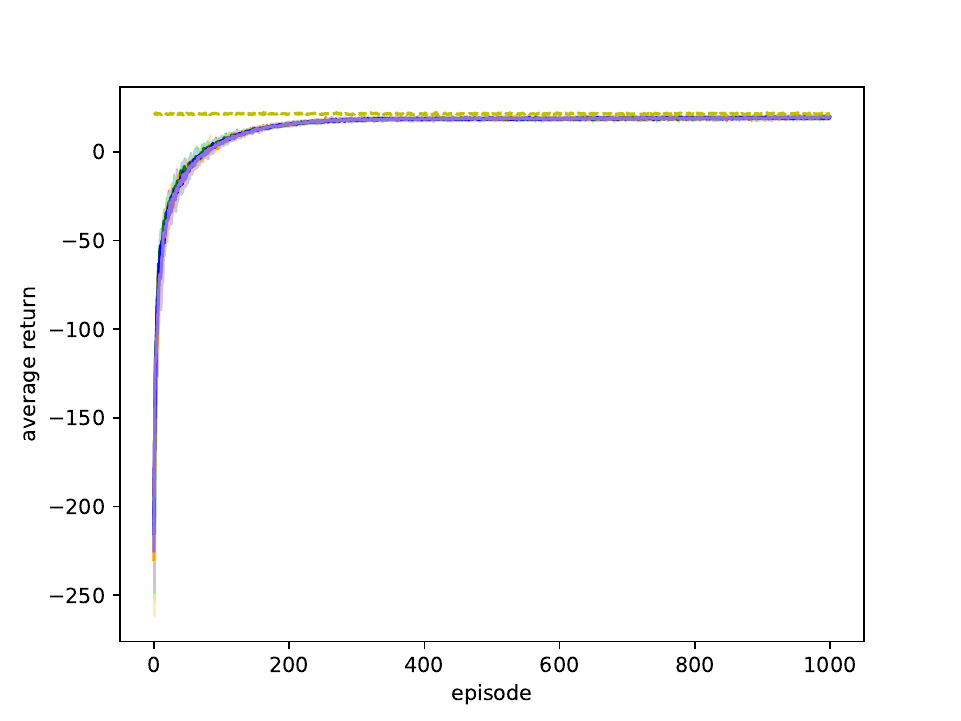}}
	\subfigure[SAC] {\includegraphics[width=.32\textwidth]{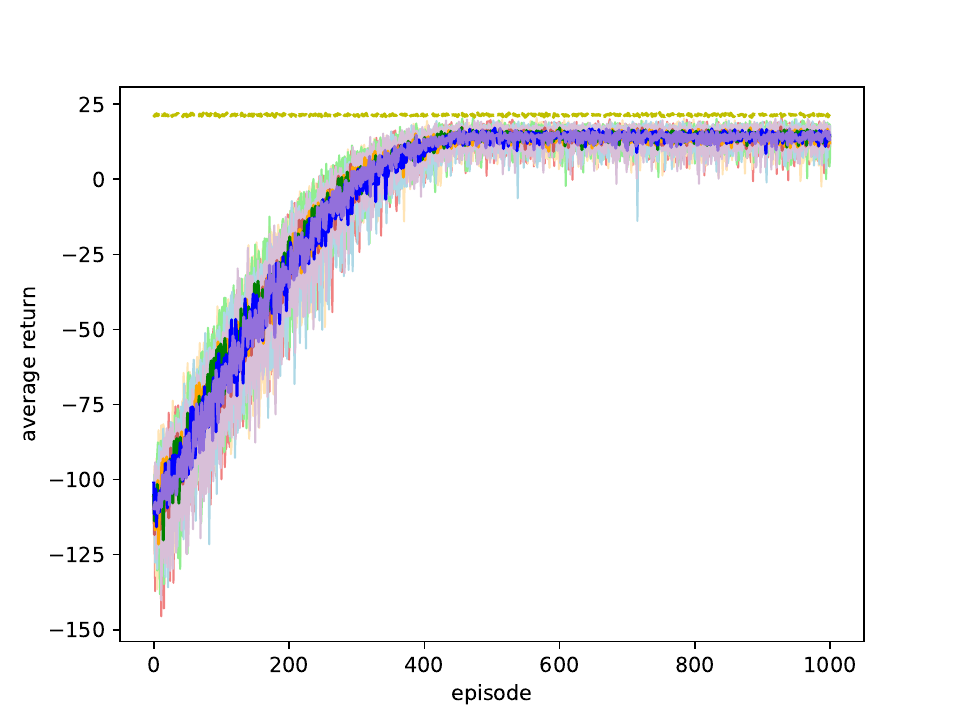}}

	\subfigure[ML-AC] {\includegraphics[width=.32\textwidth]{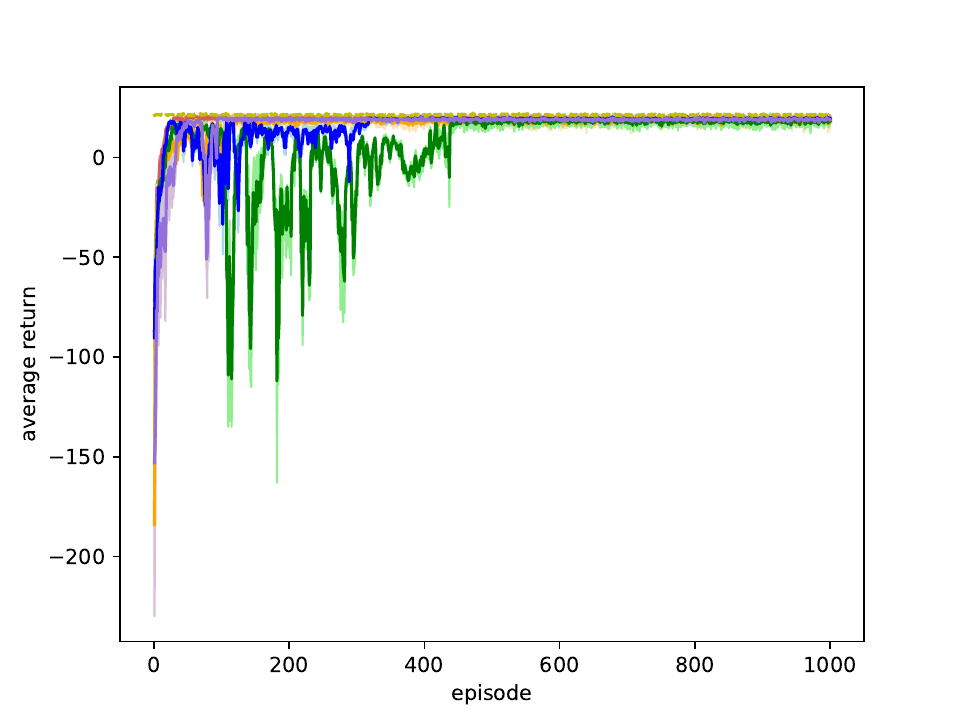}}
	\subfigure[MO-AC] {\includegraphics[width=.32\textwidth]{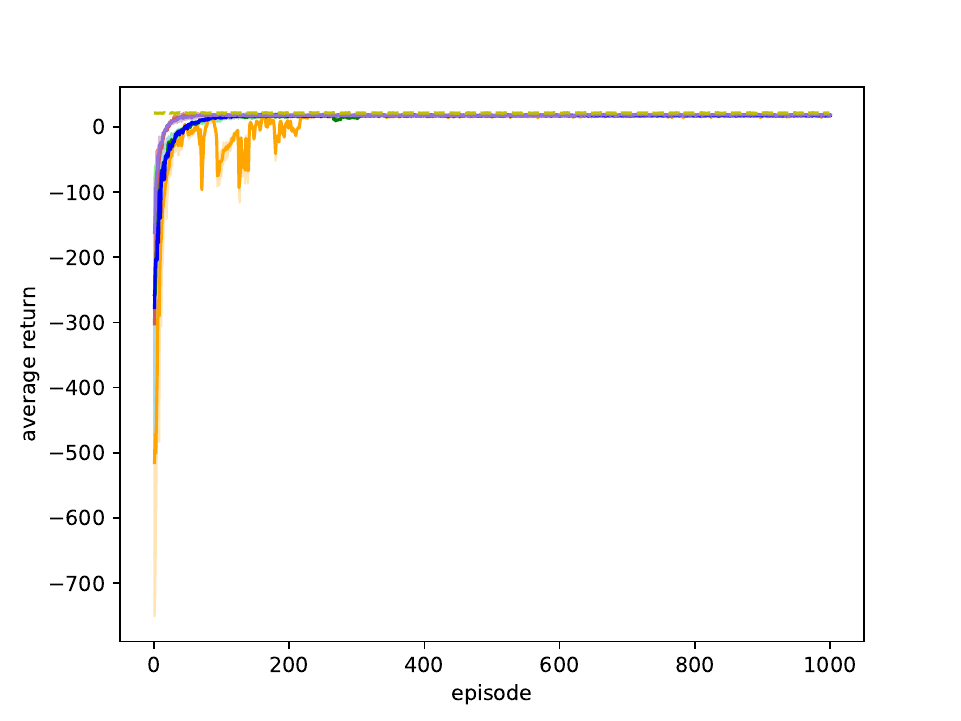}}
	\caption{Training curves in Environment 1 and 2. The different colors represent training results under different random seeds, and the yellow dashed line indicates the return gained by the optimal policy. The top five figures (Figure(a)-(e)) show the training results in Environment 1, and the bottom five figures (Figure(f)-(j)) show the training results in Environment 2. The solid curves indicates the mean value of the five out-of-sample tests at the end of each training session, and the shading covers the area between the minimum and maximum values of the five tests. }
	\label{pic:ave_return1}
\end{figure}

\begin{figure}[H]
	\centering
	\subfigure[ADP] {\includegraphics[width=.32\textwidth]{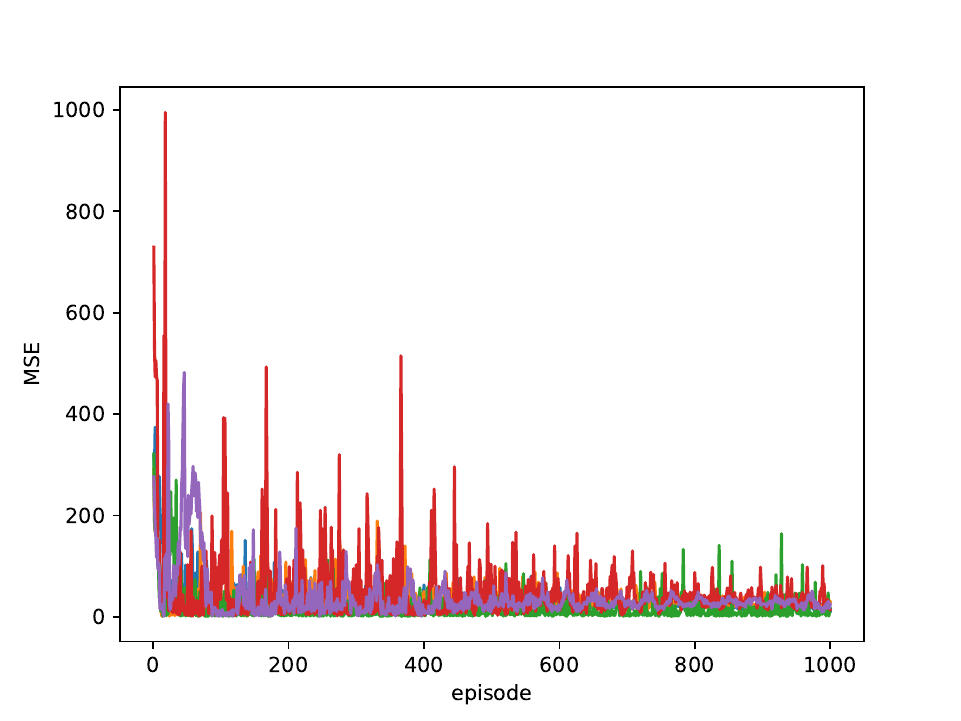}}
	\subfigure[exploratory ADP] {\includegraphics[width=.32\textwidth]{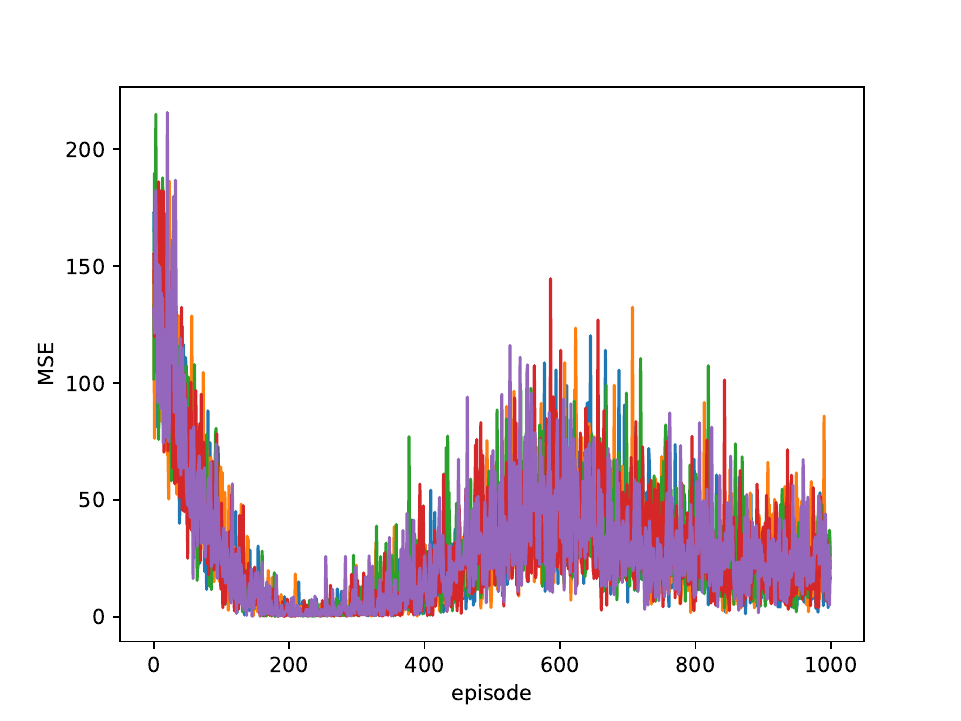}}
	\subfigure[SAC] {\includegraphics[width=.32\textwidth]{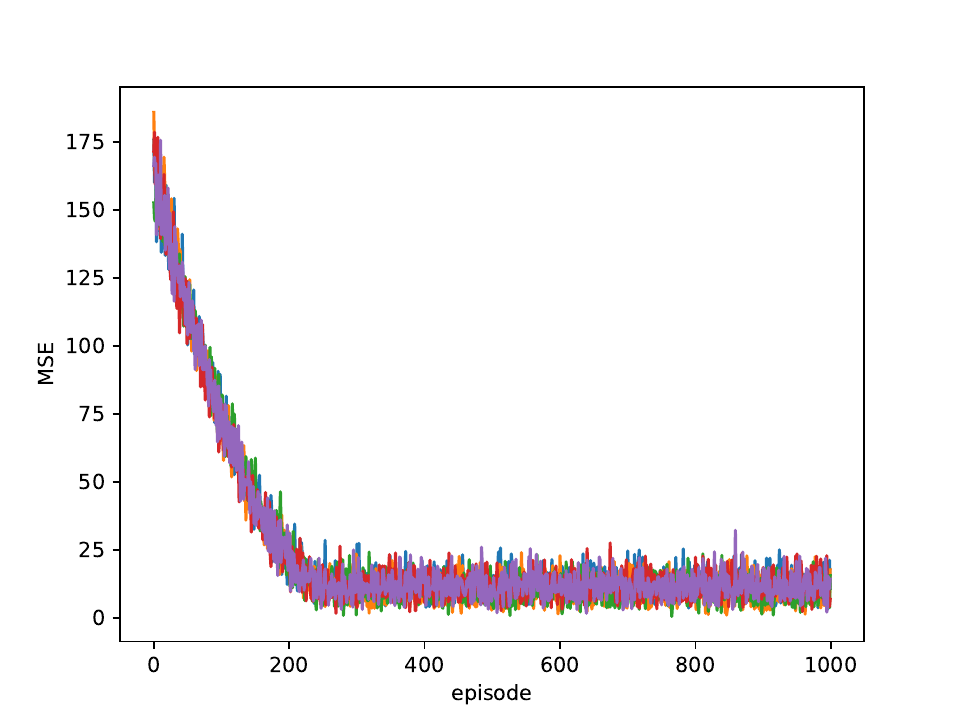}}

	\subfigure[ML-AC] {\includegraphics[width=.32\textwidth]{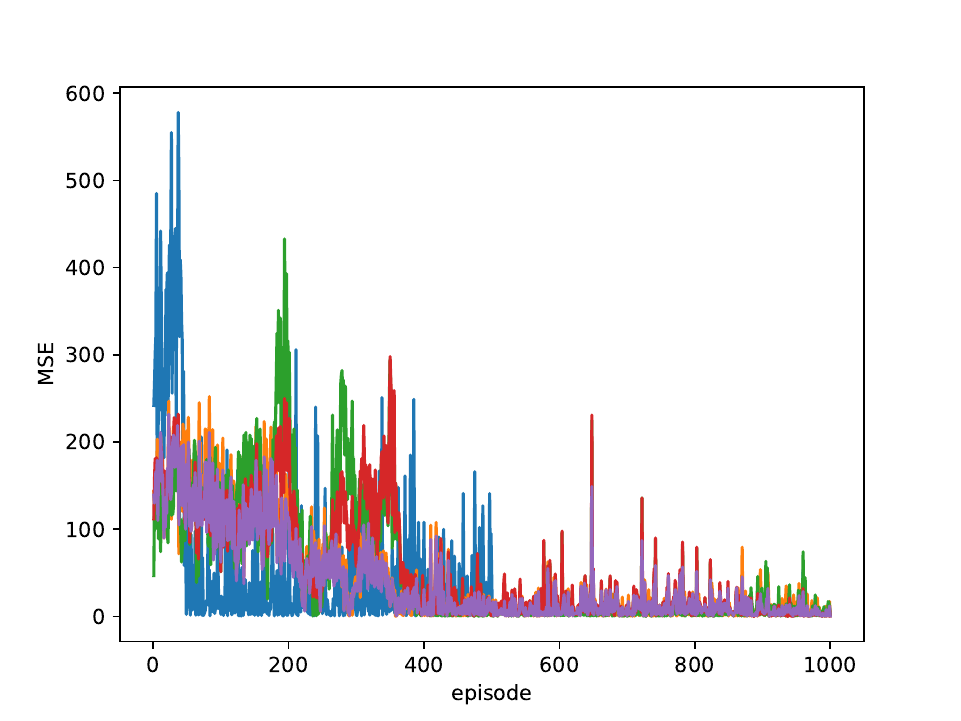}}
	\subfigure[MO-AC] {\includegraphics[width=.32\textwidth]{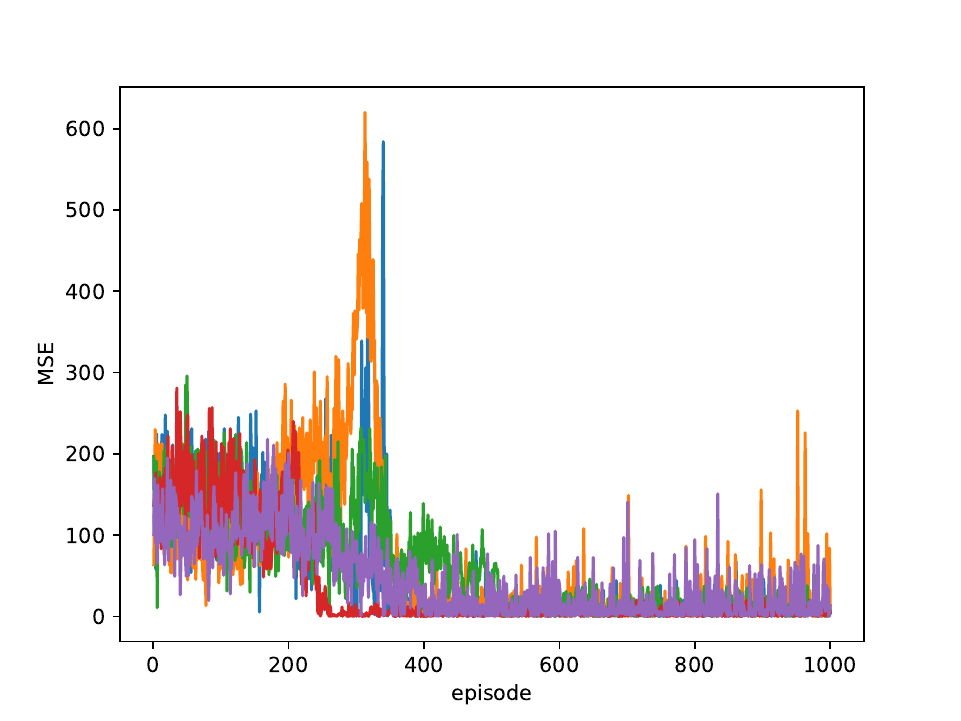}}
	
	\subfigure[ADP] {\includegraphics[width=.32\textwidth]{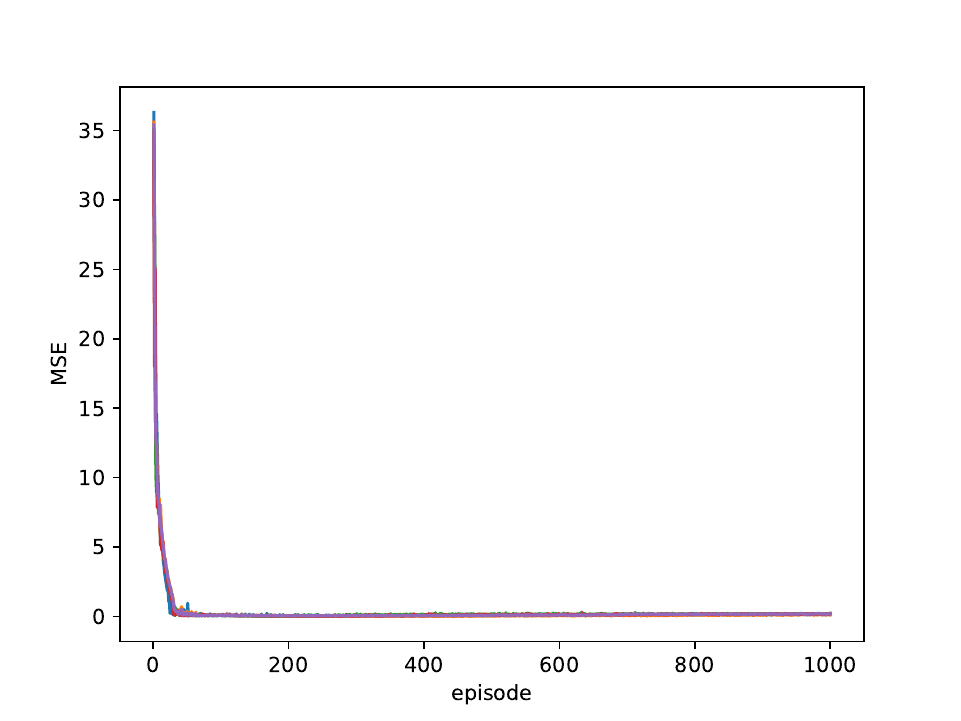}}
	\subfigure[exploratory ADP] {\includegraphics[width=.32\textwidth]{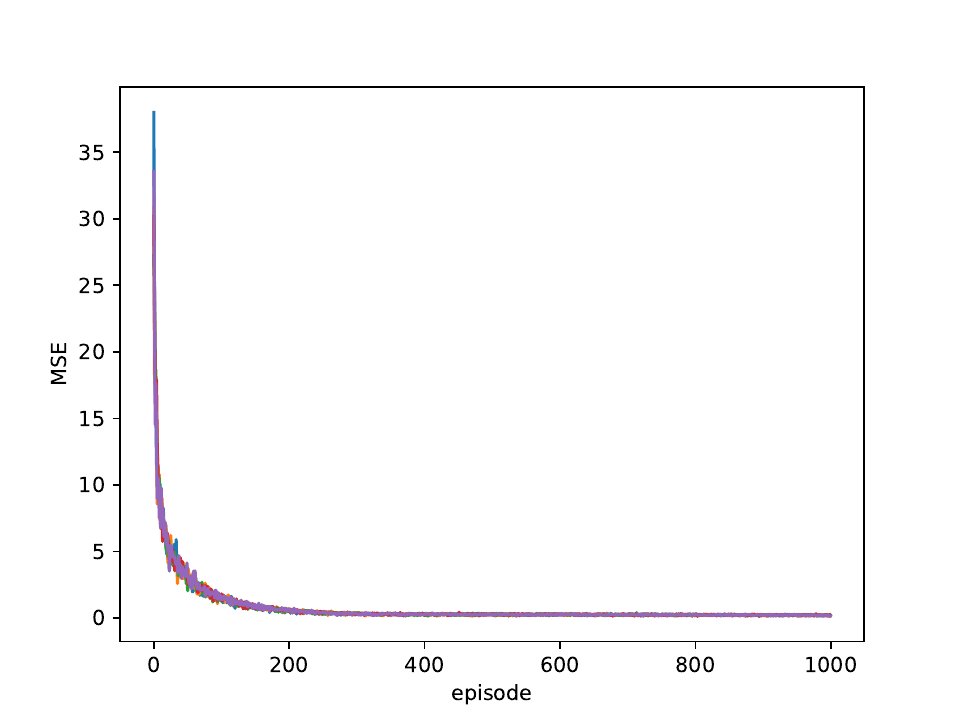}}
	\subfigure[SAC] {\includegraphics[width=.32\textwidth]{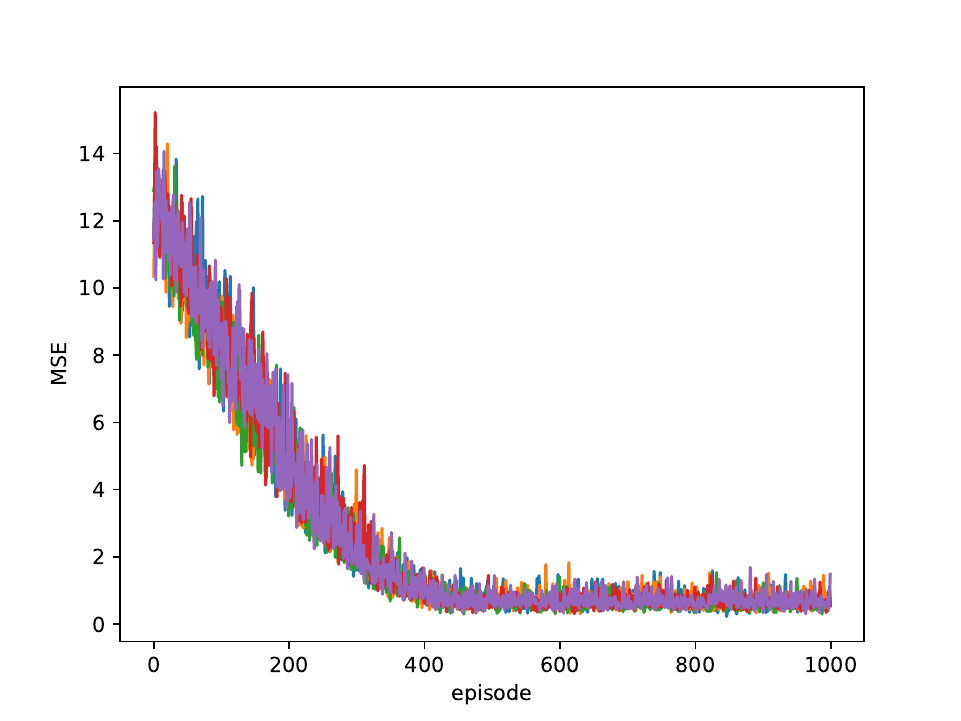}}

	\subfigure[ML-AC] {\includegraphics[width=.32\textwidth]{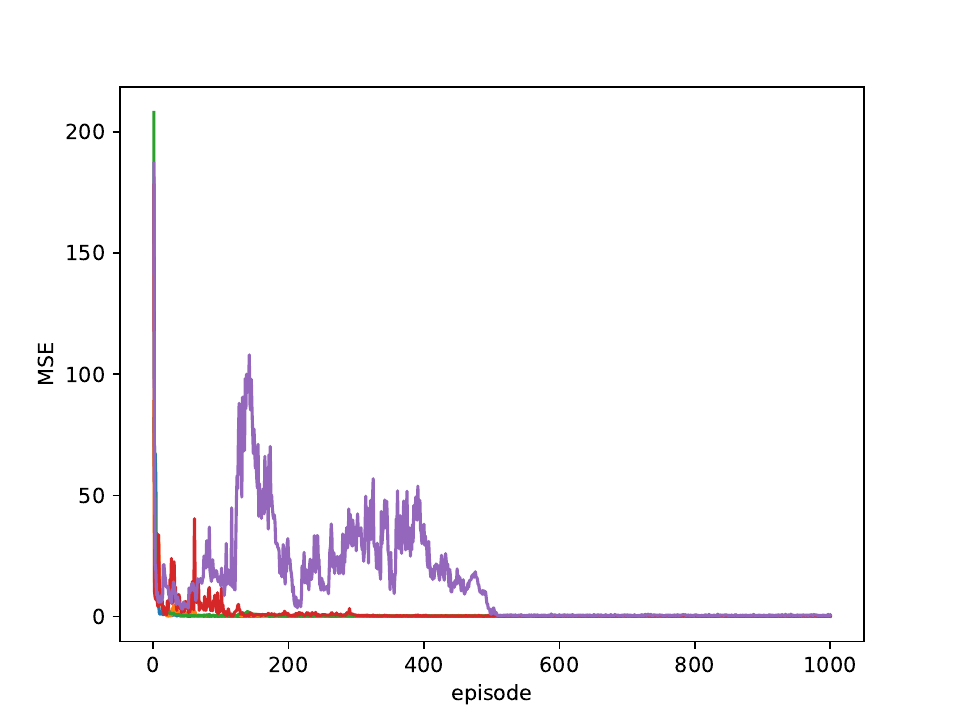}}
	\subfigure[MO-AC] {\includegraphics[width=.32\textwidth]{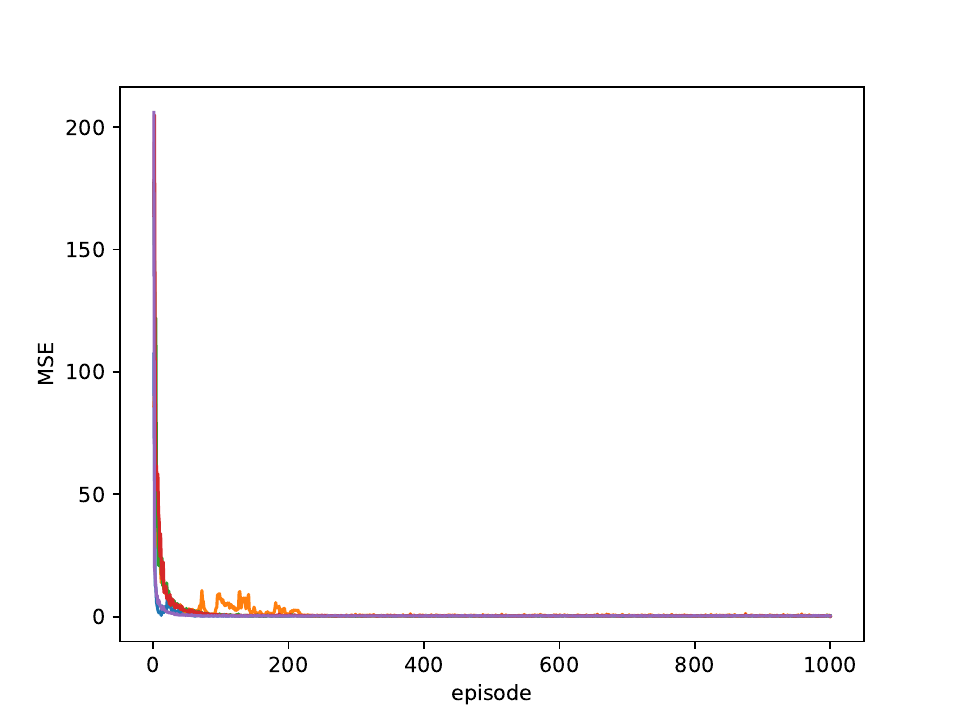}}
	\caption{MSE between the mean value and optimal policy in Environment 1 and 2. The top five figures (Figure(a)-(e)) show the training results in Environment 1, and the bottom five figures (Figure(f)-(j)) show the training results in Environment 2. The different colors represent training results under different random seeds.}
	\label{pic:error1}
\end{figure}
Training curves in Environment 1 are shown in the top five figures of Figure \ref{pic:ave_return1}. Compared to the other three algorithms, we see from Figure \ref{pic:ML-AC_return1} and Figure \ref{pic:MO-AC_return1} that average returns gained by two RL algorithms fluctuate at the beginning. The main reason is that we only use one sample to approximate the gradient, so the direction may have a certain degree of randomness. While with different random seeds, both of them converge and the average returns are close to which gained by the optimal policy. Since we obtain a stochastic policy to encourage more exploratory, the final average returns fluctuates relatively more intense compared with ADP shown in Figure \ref{pic:ADHDP_return1}, and returns gained by exploratory ADP are also more volatile because of the addition of exploration. SAC is able to learn as well, but the final returns are not as good as those obtained by MO-AC. We also calculate MSE between learned policies and the optimal policy, and the results are shown in the top five figures of Figure \ref{pic:error1}, which gives supports for Theorem \ref{thm2} since MSE tends to zero after 1000 episodes. 

Moreover, we choose TWAP as the benchmark, a static strategy that only depends on initial inventory and the length of time interval. The trading speed of TWAP is given by
\begin{equation*}
	v^{\text{TWAP}}_t=\dfrac{\mathfrak{N}}{T},\qquad \forall t\in[0,T]. 
\end{equation*}
As for different algorithms, their performance over TWAP is defined by
\begin{equation*}
	\Delta \text{PnL}=\dfrac{X_T^{\text{Agent}}-X_T^{\text{TWAP}}}{X_T^{\text{TWAP}}}. 
\end{equation*}
$ X_T $ with different superscripts represent the terminal cash gained by different policies. After training, we choose the best one according to the performance criteria $\Delta $PnL, and compare its performance with TWAP. Specifically, we conduct 100 independent tests to calculate the mean and standard deviation.

\begin{table}[!ht]
	\centering
	\caption{Comparison with benchmark: TWAP in Environment 1. The numbers in the table show the mean value of the results over one hundred out-of-sample experiments, and the standard deviations are shown in brackets. }
	\begin{tabular}{|c|c|c|c|c|c|}
		\hline
		~ & ADHDP&exploratory ADP& SAC& ML-AC & MO-AC \\ \hline
		\multirow{2}{*}{Average Return} & 12.863 &8.308&14.003& 10.354 &17.347 \\ 
		&(4.112)&(9.432)&(3.375)&(23.622)&(8.115)\\\hline
		\multirow{2}{*}{$\Delta \text{PnL}$} & -0.479&-0.663&-0.434& -0.580 & -0.298  \\ 
		~&(0.167)&(0.383)&(0.137)&(0.957)&(0.327)\\\hline
	\end{tabular}
	
	\label{benchmark1}
\end{table}

According to Table \ref{benchmark1}, MO-AC performs best among these algorithms. It shows that the designed algorithms indeed learn knowledge from the environment and perform well on the test set. This results also indicate us that in Environment 1, where the terminal penalty parameter is large, TWAP is a better choice for agents to gain more return. Actually, as demonstrated in \cite{cartea2015algorithmic}, if we assume that the agent does not affect the midprice and is adamant to clear all inventory by time $ T $, the optimal policy is to liquidate at a constant rate, which is exactly TWAP. 

The distribution of exploratory policies at each point in time is shown in Figure \ref{heat_map1}. The shade of colors represents the likelihood of taking corresponding actions, and the darker color shows the higher probability. Comparing the results of four learning algorithms with exploratory in more details, we discover that the optimal policy learned by exploratory ADP fluctuates greatly at different moments, while SAC is relatively stable. Observing the optimal policies obtained by the two RL algorithms, the closer to the terminal time, the larger the actions given by policies. It is consistent with the setting of a larger penalty parameter $\alpha$ for terminal liquidation in Environment 1, but the first two algorithms do not learn this feature. In addition, we find that MO-AC prefers a policy with less randomness. One possible reason is that MO-AC is designed to make sure the expectations of two adjacent states are equal. However, ML-AC is designed to ensure the expectation of each state and the termination state is equal. On the other hand, MO-AC algorithm could deal uncertainty better since we choose $ \xi_{t_i} $ randomly during training. 
\begin{figure}[H]
	\centering
	\subfigure[Action taken by exploratory ADP] {\includegraphics[width=.48\textwidth]{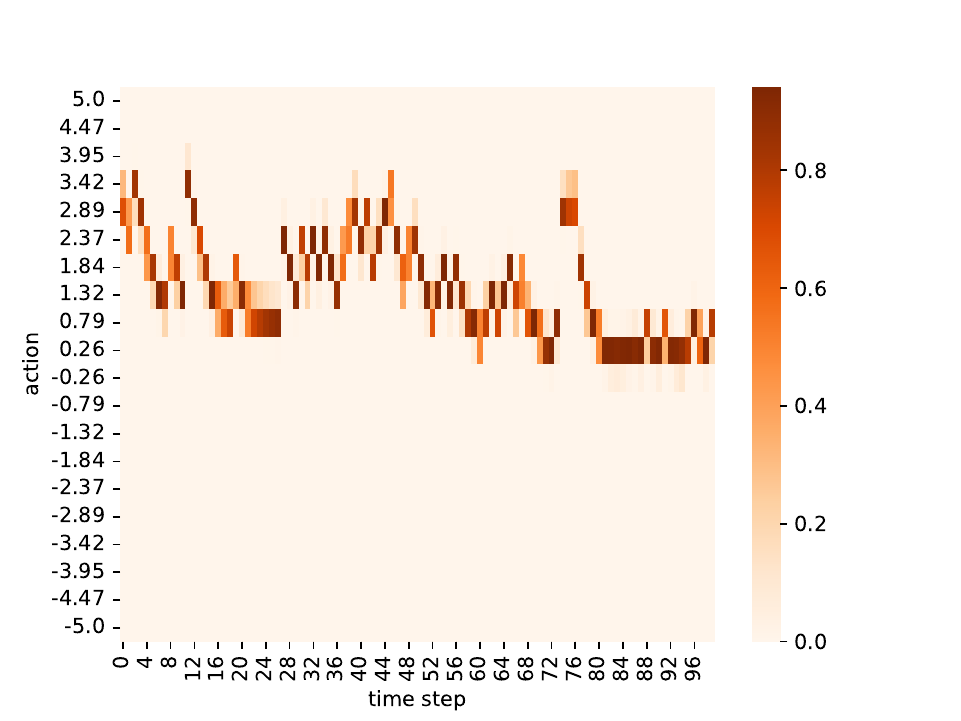}}
	\subfigure[Action taken by SAC] {\includegraphics[width=.48\textwidth]{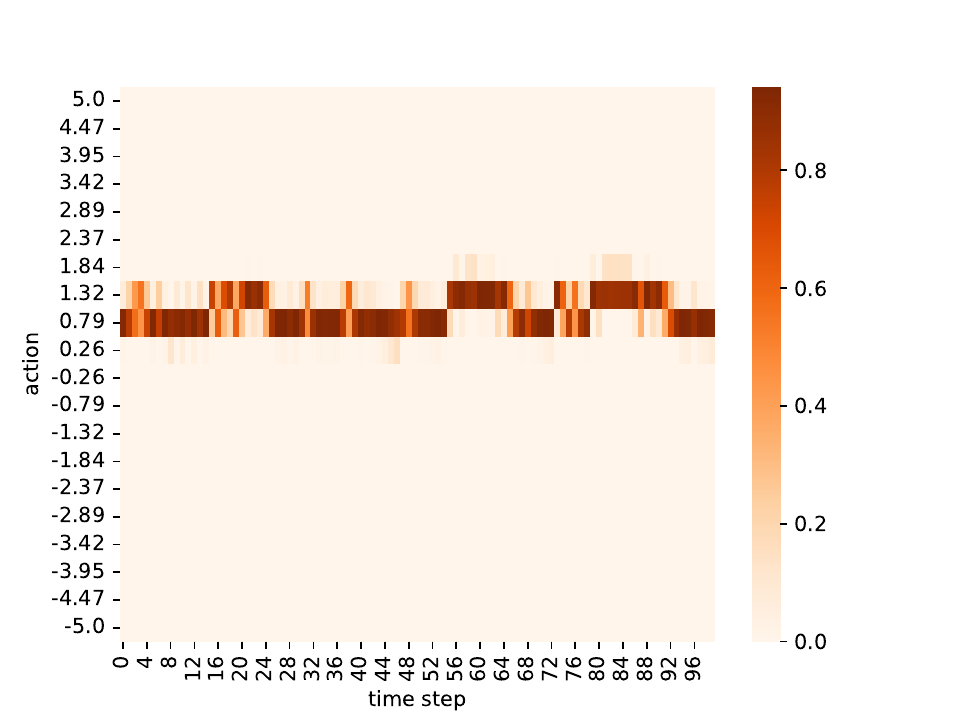}}

	\subfigure[Action taken by ML-AC ] {\includegraphics[width=.48\textwidth]{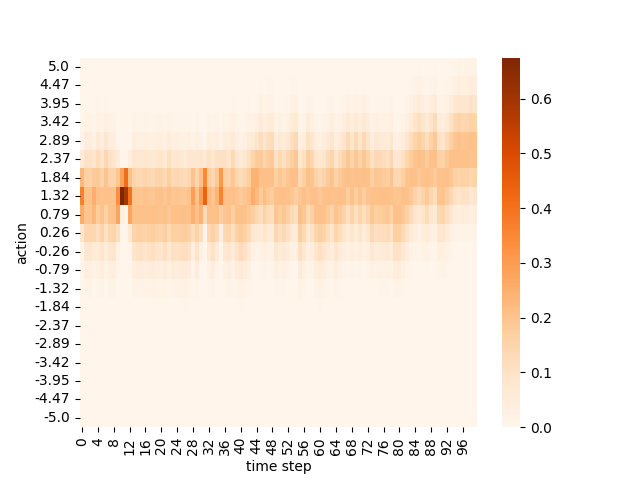}}
	\subfigure[Action taken by MO-AC ] {\includegraphics[width=.48\textwidth]{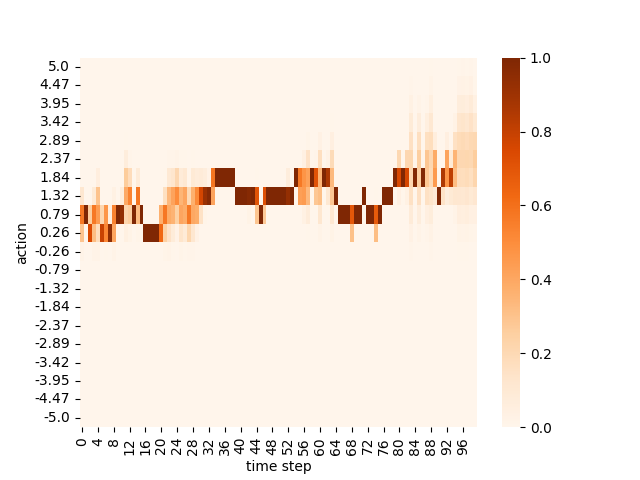}}
	\caption{The optimal exploratory policy in Environment 1. The shade of colors represents the likelihood of taking the corresponding action at that moment, with darker colors representing greater probability.}
	\label{heat_map1}
\end{figure}

The training results are rather different in Environment 2 as shown in the bottom five figures of Figure \ref{pic:ave_return1}. All of the algorithms except SAC converge quickly and the volatility decreases considerably compared with Environment 1, but the policies learned by SAC are still as volatile as the results in Environment 1. Since the price impacts in Environment 2 are larger, the agent prefers to perform more conservative to avoid large impacts. 

\begin{figure}[ht]
	\centering
	\subfigure[Action taken by exploratory ADP ] {\includegraphics[width=.48\textwidth]{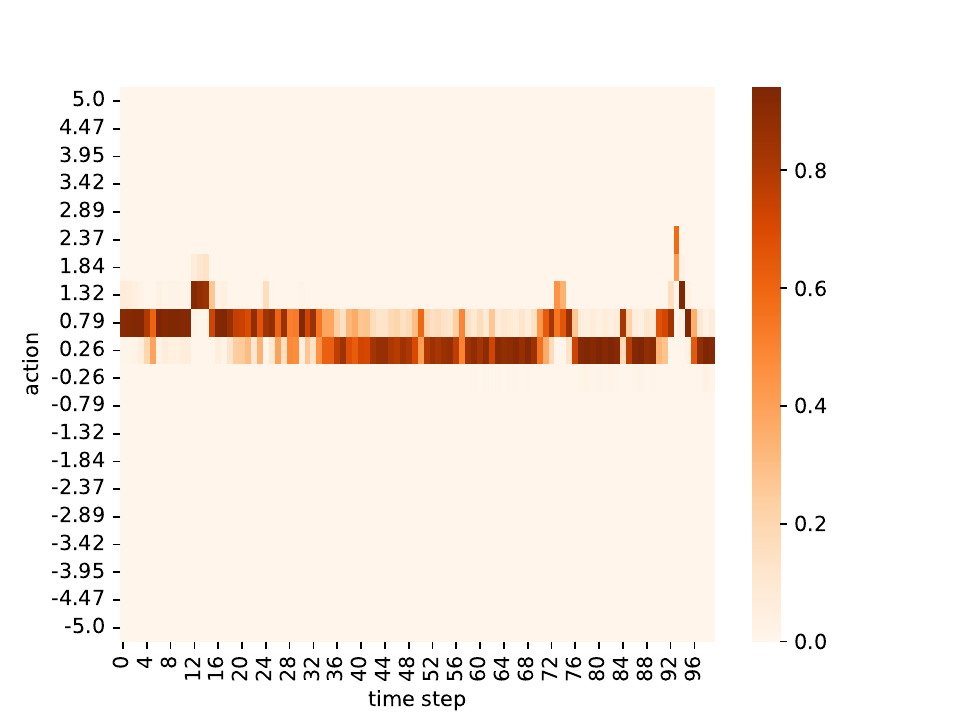}}
	\subfigure[Action taken by SAC ] {\includegraphics[width=.48\textwidth]{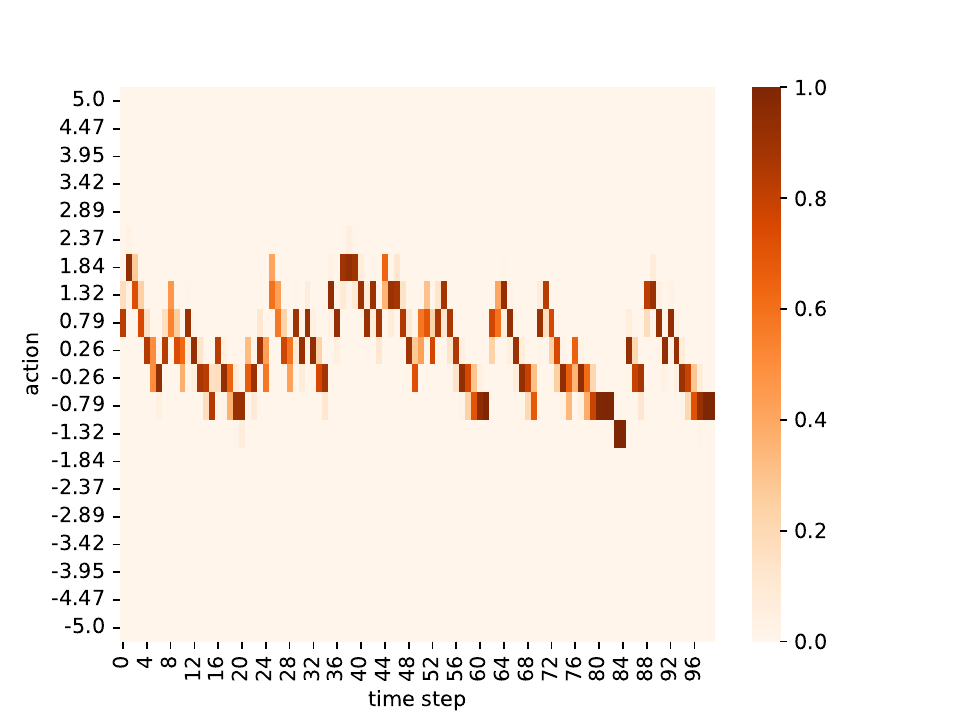}}

	\subfigure[Action taken by ML-AC ] {\includegraphics[width=.48\textwidth]{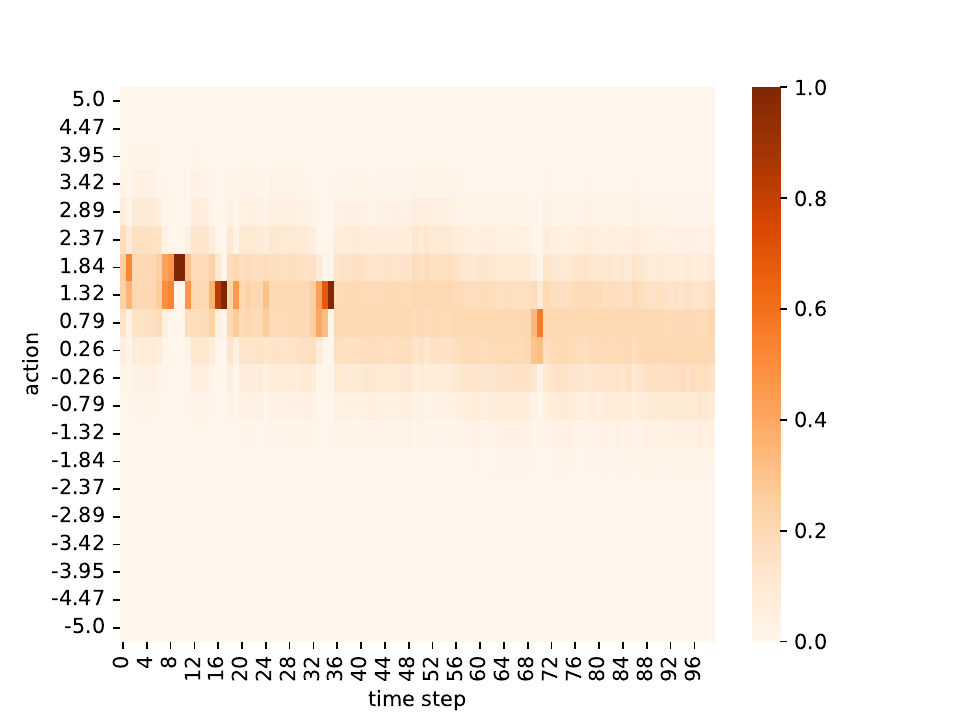}}
	\subfigure[Action taken by MO-AC ] {\includegraphics[width=.48\textwidth]{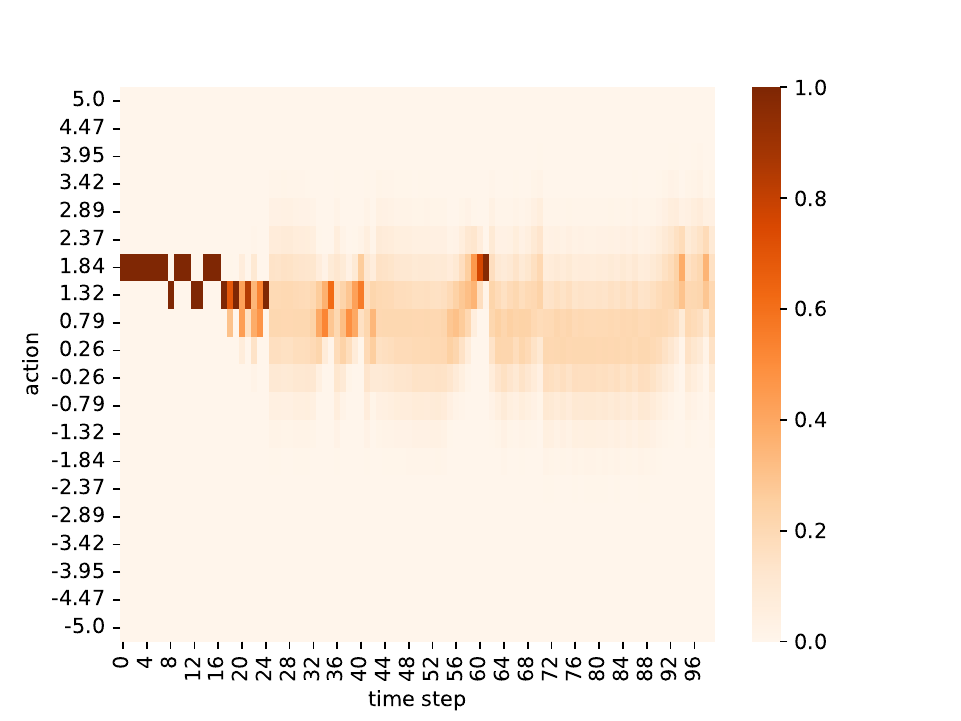}}
	\caption{The optimal exploratory policy in Environment 2. The shade of colors represents the likelihood of taking the corresponding action at that moment, with darker colors representing greater probability.}
	\label{heat_map2}
\end{figure}

\begin{table}[ht]
	\centering
	\caption{Comparison with benchmark: TWAP in Environment 2. The numbers in the table show the mean value of the results over one hundred out-of-sample experiments, and the standard deviations are shown in brackets.}
	\begin{tabular}{|c|c|c|c|c|c|}
		\hline
		~& ADP&exploratory ADP&SAC & ML-AC & MO-AC   \\ \hline
		\multirow{2}{*}{Average Return}& 19.902&19.291& 14.054 & 19.727 & 19.057  \\ 
		&(0.842)&(0.586)&(3.247)&(0.657)&(0.870)\\\hline
		\multirow{2}{*}{$\Delta \text{PnL}$}& 0.130&0.0944&-0.212 & 0.115 & 0.0784 \\ 
		~&(0.0390)&(0.0571)&(0.1572)&(0.0610)&(0.0808) \\\hline
	\end{tabular}
	\label{benchmark2}
\end{table}

Similarly, we verify the Theorem \ref{thm2} in Environment 2 according to the bottom five figures of Figure \ref{pic:error1}. The mean values of the Gaussian distributions learned by algorithms converge to the optimal policy. 

The comparison of different algorithms and TWAP in environment 2 is displayed in Table \ref{benchmark2}. The algorithms except SAC outperform TWAP stably, which indicates that our algorithms have good performance in the environment with relatively large price impacts. Besides, when we compare specific actions taken by four exploratory algorithms, which is shown in Figure \ref{heat_map2}, the optimal policy given by SAC is the most volatile. Furthermore, comparing two RL algorithms results in Figure \ref{heat_map1} and Figure \ref{heat_map2}, the agent's trading speed decreases over time in Environment 2 while increases in Environment 1. The main reason is that the terminal penalty in Environment 2 is smaller than Environment 1, so as the time is closer to $ T $, the agent prefers to retain the inventory until the terminal time $ T $, instead of facing the possible slippage that the agent has learned from previous episodes. However, exploratory ADP and SAC do not learn this feature from the environment in either setting. Similar to the results in Environment 1, we find that the policy given by MO-AC is more deterministic.

\section{Conclusions}
\label{section_conclusion}
In this paper, we focus on the exploratory liquidation problem targeting at VWAP, and develop two types of learning algorithms: adaptive dynamic programming algorithms and reinforcement learning algorithms. Our theoretical results prove the existence and uniqueness of exploratory state equations, and derive the explicit formula of the optimal policy, which is Gaussian distributed. Besides, we give some theoretical proofs for the application of continuous RL to processes with jumps. The ADP algorithms is derived from the HJB equation, and only requires estimating parameters $ b $ and $ k $. There is no need for us to estimate parameters about the market trading speed in ADP algorithms, while two RL algorithms, ML-AC and MO-AC, are totally model-free and can quickly learn from environments. Numerical results show that all algorithms are effective in different environments and outperform TWAP in the environment with larger price impacts. ADP is a good choice when the agent can model the environment and estimate the parameters $ b, k $ well, because it converges quickly. However, when the agent doesn't know much about the environment, RL algorithms provide a more general way to learn optimal policies from the environment directly through interaction.

The study of VWAP is very important because it is a price benchmark for intraday trading, and a full knowledge of VWAP can help us to design intraday trading strategies and obtain more return, which requires further research. Another possible direction for future research is to apply the learning algorithms proposed in this paper to the empirical study of Chinese market. 

\section*{Acknowledgments}
We are grateful to Sebastian Jaimungal, one of the authors of \textit{Algorithmic and High-Frequency Trading} \cite{cartea2015algorithmic}, for coming to Shanghai to present his work, and the Chinese version of the book is very popular in China. Moreover, we would like to acknowledge the support from Shanghai Yuliang Intelligence Data Technology Corporation for our VWAP research.







\end{document}